\documentclass[a4paper, 11pt, english]{article}
\usepackage[utf8]{inputenc}
\usepackage[T1]{fontenc}
\usepackage{babel}
\usepackage{tikz}
\usetikzlibrary{hobby,patterns}
\usepackage{graphicx}
\usepackage{stmaryrd}
\usepackage[a4paper]{geometry}
\usepackage{amsmath,amsfonts,amssymb,amsthm,epsfig,epstopdf,url,array}
\usepackage{rotating}
\usepackage[colorlinks=true,citecolor=red,linkcolor=blue,pdfpagetransition=Blinds]{hyperref}
\usepackage{cleveref}
\usepackage{nameref}
\usepackage{enumitem}
\usepackage{comment}
\Crefname{paragraph}{Section}{Sections}
\usepackage[normalem]{ulem}
\usepackage{fancyhdr}

\newcommand\rouge[1]{{\color{black} #1}}
\newcommand\kev[1]{{\color{red} #1}}

\usepackage{fullpage}

\newcommand{\ensemblenombre}[1]{\mathbb{#1}}
\newcommand{\N}{\ensemblenombre{N}}
\newcommand{\Z}{\ensemblenombre{Z}}

\newcommand{\R}{} 
\renewcommand{\R}{\ensemblenombre{R}}
\newcommand{\C}{\ensemblenombre{C}}

\newcommand{\T}{\ensemblenombre{T}}

\newcommand{\RE}{\mathrm{Re}}
\newcommand{\IM}{\mathrm{Im}}

\renewcommand{\leq}{\leqslant}
\renewcommand{\geq}{\geqslant}

\newcommand{\norme}[1]{\left\lVert#1\right\rVert}

\theoremstyle{plain} 
\newtheorem{prop}{Proposition}[section] 
\newtheorem{theorem}[prop]{Theorem}

\newtheorem{lemma}[prop]{Lemma}
\newtheorem{cor}[prop]{Corollary}
\theoremstyle{definition}
\newtheorem{definition}[prop]{Definition}

\newtheorem{remark}[prop]{Remark}

\newtheorem{ass}[prop]{Assumption}

\numberwithin{equation}{section}

\makeatletter
\let\original@addcontentsline\addcontentsline
\newcommand{\dummy@addcontentsline}[3]{}
\newcommand{\DeactivateToc}{\let\addcontentsline\dummy@addcontentsline}
\newcommand{\ActivateToc}{\let\addcontentsline\original@addcontentsline}
\makeatother

\begin{document}

\title{
Schr\"odinger eigenfunctions sharing 
the same modulus and applications to 
the control of quantum systems}

\author{Ugo Boscain, K\'{e}vin Le Balc'h, Mario Sigalotti
}

\maketitle

\begin{abstract}
In this paper we investigate when
linearly independent eigenfunctions of the Schr\"odinger operator 
may have the same modulus. 
General properties are established and the one-dimensional case is treated in full generality. 
The study is motivated by its application to the bilinear control of the Schr\"odinger equation. By assuming that the potentials of interaction satisfy a saturation property and by adapting a strategy recently proposed by Duca and Nersesyan, we 
discuss when the system can be steered arbitrarily fast between energy levels. Extensions of the previous results to quantum graphs are finally presented.
\end{abstract}

\kev{Updated version from June, 2025 after publication in MCSS, https://doi.org/10.1007/s00498-025-00406-y. With respect to the published version we clarified the following points:
\begin{itemize}
\item Properties of the eigenfunctions on the disk in \Cref{prop:eigendisk}.
\item Definition of the saturation cones in \eqref{eq:defHNsubspacerecurrence}.
\item Proof of the small-time isomodulus approximate controllability result in \Cref{th:smalltimepointphase}.
\item Proof of the verification of the saturation hypothesis in \Cref{lem:testdensityHinfty}.
\end{itemize}
}

\tableofcontents

\section{Introduction}

\subsection{Control of quantum systems as a motivation}

 Consider a controlled (bilinear) quantum system
\begin{equation}
i\dot\psi(t)=\left(H_0+\sum_{j=1}^m u_j(t)H_j\right)\psi(t),\qquad t \in [0,+\infty).\label{uranio}
\end{equation}
In \eqref{uranio}, 
$\psi(t)$ represents the state of the system, evolving in a Hilbert space $\mathcal H$, and $u(t) = (u_1(t), \dots, u_m(t)) \in \R^m$ is the control that can take arbitrarily large values.

One of the fundamental questions in quantum control is the following:
is it possible to induce a transition from an energy level to another one in arbitrarily small time? Since in \eqref{uranio} quantum decoherence is neglected, the model may only be applicable for small times. This is why results of controllability  in arbitrarily small time are particularly important for applications.

In the case in which the Hilbert space $\mathcal H$ where $\psi$ is evolving is finite-dimensional, this problem is well understood. Actually 
small-time 
controllability 
is possible if and only if the 
evaluation at each point 
of the unit sphere of the
Lie algebra ${\rm Lie}\{i H_1,\ldots,i H_m \}
$ generated by the (right-invariant) vector fields of the controlled part
 is of maximal rank, 
 see for instance \cite[Theorem 2]{agrachev2017note}. 
As a consequence, 
small-time
controllability 
is never possible when the operators $H_1,\ldots, H_m$ 
commute pairwise and in particular when $m=1$.

The case of an 
infinite-dimensional
Hilbert space could be very different due to the possible unboundedness of $H_0$.
This happens, for instance, 
for the 
controlled PDE 
\begin{align}
&i \partial_t \psi = - \Delta_g \psi + V(x)
\psi + \sum_{j=1}^mu_j(t) Q_j(x) \psi,\qquad t \in [0,+\infty),\label{plutonio}
\end{align}
on a Riemannian manifold $(M,g)$.  Here $\Delta_g$ is the Laplace--Beltrami operator, 
$V$ is  the \emph{internal potential}, $Q_1,\ldots Q_m$ are \emph{potentials of interaction}, and $H_0$ is the unbounded operator $-\Delta_g  + V$.
 Notice that in this case 
 we are precisely in the situation in which the multiplication operators $H_1=Q_1,\ldots, H_m=Q_m$ 
 commute pairwise.

In the last years, this subject attracted the attention of several researchers. A first example of system of the form 
\eqref{uranio} evolving in an infinite-dimensional space  with 
$m=1$ and in which transitions among energy levels in arbitrarily small time are possible has been presented in \cite{BCC12}.
Although very interesting, such an example is academic  
($i\partial_t\psi=-\left(\frac{d^2}{dx^2}\right)^{\alpha}\psi+u(t)\cos(x)\psi$,  on the circle $\mathbb S^1$, $\alpha>5/2$) and for almost ten years 
it seemed a safe bet 
that for systems 
in the form \eqref{plutonio}
such phenomenon never occurs. 
Several 
results in this direction have  been obtained  (see, 
for  instance, \cite{beauchard2014minimal,BCT18}
 and \cite{BBS21}). 

A surprising result going in the opposite direction was provided more recently by Duca and Nersesyan \cite{DN21}, who proved that on the $d$-dimensional torus, for systems of the type \eqref{plutonio} with $V=0$, certain transitions among 
energy levels
are possible in arbitrarily small time. In particular it was proven that it is possible to reach states of arbitrarily large energy in arbitrarily small time.
This result has been obtained under an assumption on the potentials of interaction called  saturation hypothesis (see Section~\ref{sec:bilinearquantumsystemsmanifolds} for more details).
Such type of hypothesis  was inspired by a similar one introduced by Agrachev and Sarychev for the control of the Navier--Stokes equations in \cite{AS05} and \cite{AS06}. 

The result from \cite{DN21} was 
extended and proven in a simpler form in \cite{CP22}. In particular in that paper examples of systems evolving on manifolds with a topology different from that of the torus were studied. After the appearance on ArXiv of the first version of the present paper, the first examples of small-time globally approximately controllable bilinear Schrödinger equations were exhibited in the preprint \cite{BP24}.

The techniques developed in \cite{DN21} and \cite{CP22} actually permit to prove that (under the  saturation hypothesis)    it is possible to  steer in arbitrarily small time a state $\psi(x)$ arbitrarily close to a state of the form $e^{i\theta(x)}\psi(x)$ for 
some real-valued function $\theta(x)$. 

Then the possibility of making an arbitrarily fast transition among  different energy levels is reduced to the problem of {\em the existence of eigenfunctions corresponding to different energy levels that share the same modulus}. 
This is the property that we study in this paper.

To clarify the situation, let us discuss briefly the case of the circle $\R/2\pi\Z$. The eigenvalues of  $H_0=-\frac{d^2}{d x^2}$ are $\lambda_k=k^2$, $k=0,1,2,\ldots$ The fundamental state  ($k=0$) is non-degenerate and corresponds to the real \rouge{constant} eigenfunction $\phi_0\equiv \frac1{\sqrt{2 \pi}}$. The other eigenvalues ($k\geq1$) are  double and correspond to the real eigenfunctions   $\phi_k^c(x)=\frac1{\sqrt{\pi}}\cos(k x)$ and $\phi_k^s(x)=\frac1{\sqrt{\pi}}\sin(k x)$.
Although these eigenfunctions  do not differ by a factor of phase only, a suitable linear combination in each eigenspace permits to construct eigenfunctions all with modulus $1/\sqrt{2\pi}$: 
\begin{equation}
    \phi_k^\pm(x)= \frac{1}{\sqrt{2}}(\phi_k^c(x)\pm i \phi_k^s(x)) = \frac1{\sqrt{2 \pi}}e^{\pm i k x},\qquad k=1,2,\ldots
    \label{stronzio}
\end{equation}

With the technique developed by Duca and Nersesyan, it is possible to induce a transfer of population in arbitrarily small time among the eigenfunctions \eqref{stronzio}. This can be done approximately and  suitable  potentials of interaction are necessary. 

The purpose of this paper
is to study whether
the presence of linearly independent eigenfunctions sharing the same modulus is a specific feature of this problem or can occur in more general situations.

\rouge{\subsection{Organization of the paper and main results}

The structure of the paper is the following. 

In Section \ref{s-basic} we give some basic definitions and examples concerning Schrödinger operators having eigenfunctions sharing the same modulus. We focus on the torus, the sphere, the unit disk, and the Euclidean space. For the disk $\mathbb D$, we notably prove that two eigenfunctions of different energy levels cannot share the same modulus while it is possible to find two independent eigenfunctions of the same energy level that share the same modulus, see \Cref{prop:eigendisk}. We then introduce basic concepts about small-time approximate controllability, we introduce the saturation hypothesis and the result by Duca and Nersesyan \cite{DN21}. We then state in \Cref{th:smalltimepointphase} a new result  concerning small-time approximate controllability between states sharing the same modulus adapted to manifolds with boundary, in the spirit of the ones given by Duca and Nersesyan for the torus  \cite{DN21} and Chambrion and Pozzoli for manifolds without boundary \cite{CP22}.   This result leads to the possibility of making small-time transitions among states having the same energy for the disk $\mathbb D$, see \Cref{cor:transitiondisk} below. 

In Section~\ref{s-moduli} we study  conditions for Schr\"odinger operators to admit eigenfunctions sharing the same modulus in the general case and in the one-dimensional case (Section \ref{ss-1}).
Our main result in the one-dimensional case is Theorem~\ref{thm:1D} that we reformulate here for convenience.

\begin{theorem}
\rouge{Let $M$ be a smooth and connected one-dimensional manifold (possibly with boundary) equipped with a Riemannian metric $g$. In $L^2(M)$, consider the Schrödinger operator $H_0=-\Delta_g^2+ V$ (with $V\in L^{\infty}_{\rm loc}(M)$ bounded from below and Neumann or Dirichlet boundary conditions if $M$ has boundary) that we assume to be self-adjoint and with compact resolvent. 
If $H_0$ admits two 
$\C$-linearly independent
eigenfunctions $\phi_k$ and $\phi_\ell$
sharing the same modulus, 
then necessarily $M$ is 
a closed curve and $\phi_k$, $\phi_\ell$ are nowhere vanishing on $M$. If, moreover, 
$\phi_k$, $\phi_\ell$
correspond to distinct eigenvalues, then $V$ is constant. }
\end{theorem}

One can read this result in the following way: the requirement of the existence of eigenfunctions sharing the same modulus for 
{\color{black} two different} 
eigenvalues gives a topological constraint (the manifold should be the circle) and a condition on the  potential $V$ that must be constant. 
However if we are interested in transitions in between the same eigenspace, then 
non-constant potentials are admitted. See Remark 
\ref{rem:exp} for an explicit example.

Beside the two-dimensional torus, the sphere $\mathbb{S}^2$ and the disk $\mathbb D$,
the study of other systems in  dimension 2 is complicated by the lack of explicit expressions for the eigenfunctions. 

We then study (Section \ref{s-grafi}) problems that in a sense are intermediate between dimension 1 and 2, namely problems on {\color{black}one-dimensional graphs. 
Such} graphs are obtained by gluing together one-dimensional manifolds, but their topology is much richer than the topology of one-dimensional manifolds and for certain properties they are more similar to higher dimensional manifolds. 
See, for instance, the survey \cite{Kuc02} studying spectral properties in thin structures collapsing to a graph. Note that we also establish a new result  concerning small-time approximate controllability between states sharing the same modulus adapted to quantum graphs (\Cref{th:smalltimeapproximatequantumgraphs}).

On graphs we exhibit some new phenomena. We show the existence of graphs for which

$\bullet$ there exists an infinite sequence of eigenfunctions corresponding to arbitrarily high energy levels that share the same modulus, but this is not the case for all 
energy levels. As a consequence, with an appropriate choice of potentials of interactions, we are able to prove a small-time approximate controllability result between these eigenstates (\Cref{th:resultisographeight} in
Section~\ref{subsection-graph-8}); 

$\bullet$ 
all energy levels but the ground state are degenerate, but there exists no pair of eigenfunctions sharing the same modulus and having different energy levels (Section \ref{subsection-graph-3branches}).}

 \medskip

{\bf Acknowledgements.} This work has been partly supported by the ANR project TRECOS ANR-20-CE40-0009, by the ANR-DFG project ``CoRoMo'' ANR-22-CE92-0077-01, and has received financial support from the CNRS through the MITI interdisciplinary programs.

\section{Basic setting}
\label{s-basic}
\subsection{The Schr\"odinger operator on a Riemannian manifold}

Let $M$ be a smooth and  connected manifold of dimension $d$, possibly with boundary, equipped with a \rouge{smooth} Riemannian metric $g$. We assume that $M$, endowed with the Riemannian distance, is complete as a metric space. Let $\Delta_g = \text{div}_{\omega_g} \circ \nabla_g$ be the Laplace--Beltrami operator on $(M,g)$, where $\text{div}_{\omega_g}$ is the divergence operator with respect to the Riemannian volume and $\nabla_g$ is the Riemannian gradient. When $\partial M \neq \emptyset$, we split $\partial M = \cup_{j \in I} \mathcal C_j$ where $\{C_j\mid j \in I\}$ is the set of the connected components of $\partial M$.

We define the Hilbert space $\mathcal H = L^2(M;\C)$, where integration is considered with respect to the Riemannian volume. 
For simplicity, in the following we 
write $L^2(M)$ for $L^2(M;\C)$ and, similarly, $H^k(M)$ for $H^k(M;\C)$, $k\ge 1$. Let $V \in L^{\infty}_{\rm loc}(M;\R)$ be bounded from below, $H_0 = - \Delta_g + V$, and 
define
\begin{equation}
\label{eq:defdomH0}
    \text{Dom}(H_0) := \{\psi \in H^2(M)\mid {\color{black}H_0\psi=}-\Delta_g \psi + V \psi \in L^2(M)\ \text{and}\ P_j(\psi)= 0\ 
    \text{on}\ \mathcal{C}_j\ \forall j \in I\},
\end{equation}
where, for every $j\in I$, $P_j(\psi)$ is either the trace of $\psi$ on $ \mathcal{C}_j$ or the trace of  $\partial_{\nu} \psi$ on 
$ \mathcal{C}_j$ with $\nu$ the outer unit normal vector to $M$.
In the sequel, we split $I = I_D \cup I_N$, where $I_D$ (respectively, $I_N$) corresponds to the set of indices such that $P_j(\psi)$ is the trace of $\psi$ on $\mathcal C_j$ (respectively, 
the trace of $\partial_{\nu} \psi$ on $\mathcal C_j$).

\begin{ass}
\label{ass:op}
    The unbounded Schr\"odinger operator $(H_0 ,\text{Dom}(H_0))$, {\color{black}where {\color{black}$H_0$ and } $\text{Dom}(H_0)$ are as in \eqref{eq:defdomH0},} 
    is  self-adjoint on $L^2(M)$ with compact resolvent.
\end{ass}
As a consequence of \Cref{ass:op}, from the spectral theorem (see, e.g., \cite[Theorem 6.2]{CR21}), there exists an orthonormal basis of $L^2(M;\C)$ composed of eigenfunctions $(\phi_k)_{k \geq 1}$ associated with the sequence of real eigenvalues $(\lambda_k)_{k \geq 1}$, that is for every $k \geq 1$, $\phi_k \in E_{\lambda_k} = \mathrm{Ker}(H_0-\lambda_k I)$. Let us first give some examples, that are mainly taken from \cite{Cha84}, on which we will focus in the following.
\begin{enumerate}
    \item Let $M = \T^d = (\R / 2 \pi \Z)^d$ be the $d$-dimensional torus, $\Delta_g$ be the Euclidean Laplace operator, $V \in L^{\infty}(\T^d;\R)$. The unbounded operator $(-\Delta_g + V, H^2(\T^d))$ satisfies Assumption~\ref{ass:op}.
    \item Let $M = \mathbb S^2$ be the two-dimensional sphere equipped with the standard Riemannian metric and  $\Delta_g$ be the corresponding  Laplace--Beltrami operator.
Let $V \in L^{\infty}(\mathbb S^2;\R)$. The unbounded operator  $(-\Delta_g + V, H^2(\mathbb S^2))$ satisfies Assumption~\ref{ass:op}. 
\item Let $\Omega$ be a smooth bounded connected domain of $\R^d$, $\Delta_g $ be the Euclidean Laplace operator, $\nu$ be the outer unit normal vector to $\partial \Omega$ and $V \in L^{\infty}(\Omega;\R)$. The unbounded operators $(-\Delta_g + V, H^2(\Omega) \cap H_0^1(\Omega) )$ and $(-\Delta_g + V, \{ \psi \in H^2(\Omega)\mid \partial_{\nu} \psi = 0\ \text{on}\ \partial \Omega\})$ satisfy Assumption \ref{ass:op}.
\item Let $M = \R^d$, $\Delta_g $ be the Euclidean Laplace operator, $V \in L_{\rm loc}^{\infty}(\R^d;\R)$
be such that $\lim_{\|x\| \to +\infty} V(x) = +\infty$.  Then the unbounded operator $(-\Delta_g + V, \{ \psi \in H^2(\R^d)\mid V \psi \in L^2(\R^d)\})$ satisfies Assumption \ref{ass:op}. 
\end{enumerate}

\begin{remark}
Assumption \ref{ass:op}  does not cover all interesting situations.  Consider for instance the Schrödinger operator 
\[
-\frac{d^2}{dx^2} +\frac{c}{x^2(1-x)^2},\qquad x\in (0,1).
\]
If $c\geq 3/4$ this operator is essentially self-adjoint in $L^2((0,1))$ (this is a consequence of \cite{ReedSimon}, Theorem X.10) and hence 
this situation is covered by Assumption~\ref{ass:op} (even if 
it does not satisfy the standing assumptions, since the metric space $(0,1)$ endowed with the Euclidean distance is not complete).
If $c< 3/4$ this operator is not essentially self-adjoint. Self-adjoint extensions exist, but they do not correspond to boundary conditions of Dirichlet or Neumann type. In other words the latter case is not covered by Assumption~\ref{ass:op}.
\end{remark}

\subsection{Eigenfunctions sharing the same modulus: definition and examples
}

The following definition 
introduces rigorously  
the key notion of 
eigenfunctions of the operator $H_0 = - \Delta_g + V$ sharing the same modulus. 

\begin{definition}
\label{def:pointphase}
    For $k, \ell \geq 1$, two $\C$-linearly independent eigenfunctions $\phi_k \in E_{\lambda_k}$ and $\phi_\ell \in E_{\lambda_\ell}$ are said 
 to \emph{share the same modulus}
 if 
    \begin{equation}
        |\phi_k(x)| =         |\phi_\ell(x)|,\qquad \mbox{for every }x \in M.
    \end{equation}
\end{definition}
We are interested in different occurrences of eigenfunctions 
sharing the same modulus:
\begin{itemize}
    \item For $k \geq 1$, $H_0$ may admit \emph{eigenfunctions  sharing the same modulus inside the energy level $\lambda_k$}, that is, there may exist two $\C$-linearly independent
    eigenfunctions in $E_{\lambda_k}$ that  share the same modulus;
    \item
    For $k,\ell \geq 1$,     $H_0$ may admit \emph{two eigenfunctions $\phi_k \in E_{\lambda_k}$ and $\phi_{\ell} \in E_{\lambda_{\ell}}$ sharing the same modulus and corresponding to different energy levels $\lambda_k$ and $\lambda_{\ell}$;}
    \item $H_0$ may admit \emph{eigenfunctions sharing the same modulus corresponding to all energy levels}, that is, there may exist a subsequence $(\phi_{k_j})_{j\ge 1}$ of an orthonormal basis of eigenfunctions $(\phi_{k})_{k\ge 1}$ 
     such that 
     the functions $\phi_{k_j}$ all
     share the same modulus and such that 
     $\{\lambda_k\mid k\ge 1\}=\{\lambda_{k_j}\mid j\ge 1\}$. 
\end{itemize}

Let us illustrate these notions  on some particular examples of Riemannian manifold $M$ and operator $H_0 = - \Delta_g + V$. \\

$1.$ \rouge{Let $M=\T^d$, $V=0$, and $H_0 = -\Delta$. The eigenvalues are given by
    \begin{equation}
        \label{eq:eigenvaluestorus}
        \lambda_k = n_1^2 + \dots + n_d^2,\qquad (n_1, \dots, n_d) \in \N^d,
    \end{equation}
    and for such a $d$-tuple $n=(n_1, \dots, n_d) \in \N^d$, {\color{black}the functions 
    \begin{equation}
    \label{eq:eigenfunctiontorus}
    \Phi_{k,n}^s(x) = \prod_{j=1}^d e^{i s_j n_j x_j}  ,\ \qquad s = (s_1, \dots, s_d) \in \{-1,+1\}^d,\  x = (x_1, \dots, x_d) \in \T^d
    \end{equation}
    are (complex) eigenfunctions corresponding to the eigenvalue $\lambda_k$. Not all such functions are necessarily linearly independent, since some $n_j$ may be equal to $0$. The eigenspace corresponding to $\lambda_k$ is spanned by all functions $\Phi^s_{k,n}$
    with $n$ satisfying \eqref{eq:eigenvaluestorus}. }   
    Note that $\mathrm{dim}(E_0) = 1$ and for every $\lambda_k >0$, $\mathrm{dim}(E_{\lambda_k}) \geq 2$. Indeed, for a given $d$-tuple $n=(n_1, \dots, n_d) \in \N^d \setminus \{0\}$ satisfying \eqref{eq:eigenvaluestorus}, the following two eigenfunctions are linearly independent
     \begin{equation}
    \label{eq:eigenfunctiontorusPM}
    \Phi_{k,n}^{\pm}(x) = \prod_{j=1}^d e^{i \pm n_j x_j}  ,\qquad x = (x_1, \dots, x_d) \in \T^d.
    \end{equation}
\begin{prop}
    The operator $H_0$ admits eigenfunctions sharing the same modulus 
    inside each energy level $\lambda_k > 0$. Moreover the operator $H_0$ admits
  eigenfunctions sharing the same modulus corresponding to all energy levels.
\end{prop}
\begin{proof}
Let $\lambda_k >0$ be a positive eigenvalue, then for every $n=(n_1, \dots, n_d) \in \N^d$ such that \eqref{eq:eigenvaluestorus} holds, for every $s = (s_1, \dots, s_d) \in \{-1,+1\}^k$, $\Phi_{k,n}^s$ has a constant modulus equal to $1$. This proves that $H_0$ admits $\mathrm{dim}(E_{\lambda_k}) \geq 2$ linearly independent eigenfunctions sharing the same modulus 
    inside $\lambda_k$.
    
    For $\lambda_k \neq \lambda_l$, we also have $| \Phi_{k,n}^s | = | \Phi_{l,n'}^{s'} | = 1$ so $H_0$ admits eigenfunctions sharing the same modulus corresponding to all energy levels.
\end{proof}

With this at hand, it is by now easy to construct manifolds $M$ and operators $H_0$ that have eigenfunctions sharing the same modulus. Indeed, let us set $M = \T^{d_1} \times M_2$ with $d_1 \in \N^*$ and $M_2$ a smooth, connected, complete manifold of dimension $d_2 \in \N^*$, possibly with boundary, where $d_1 + d_2 = d$. 
The eigenvalues 
of the  operator $H_0 = - \Delta_{x_1,x_2}$ on $M$
are then given by the sum of eigenvalues of $-\Delta_{x_1}$ in $\T^{d_1}$ and $-\Delta_{x_2}$ in $M_2$ and the eigenfunctions are given by the product of associated eigenfunctions. Then, let us define the eigenstates
\begin{equation*}
       \phi_{k,n}^{s_1}(x) = \Phi_{k_1,n_1}^{s_1}(x_1) \phi_{r_2}(x_2)  ,\qquad x = (x_1,x_2) \in \T^{d_1} \times M_2,
 \end{equation*}
    where $\Phi_{k_1,n_1}^{s_1}$ is defined in \eqref{eq:eigenfunctiontorus} and $\phi_{r_2}$ is an eigenfunction of $-\Delta_{x_2}$ on $M_2$ associated with the level $\lambda_{r_2}$. For two different $s_1, s_1' \in \{-1,+1\}^{d_1}$,  we have that $\phi_{k,n}^{s_1}, \phi_{k,n}^{s_1'}$ are 
    eigenfunctions sharing the same modulus inside the same energy level $\lambda_k = \lambda_{k_1} + \lambda_{r_2}$. In the same way, one can also exhibit eigenfunctions sharing the same modulus correspond to different energy levels of the type $\lambda_k = \lambda_{k_1} + \lambda_{r_2}$ and $\lambda_\ell = \lambda_{\ell_1} +  \lambda_{r_2}$. Notice that such a construction is not 
    specific to the torus $\T^{d_1}$ and it can performed on product manifolds $M = M_1 \times M_2$  for which the operator $-\Delta_{x_1}$ admits eigenfunctions sharing the same modulus on $M_1$.}\\

$2.$ Let $M = \mathbb S^2$, $V=0$, and $H_0 = -\Delta_g$. 
Denote by $Y_m^l$, $l \in \N$, $m \in \{-l,\dots, l\}$, the spherical harmonics, i.e., the eigenfunctions of $-\Delta_{\mathbb S^2}$.
Then the eigenvalue 
associated with  $Y_m^l$ is $l(l+1)$
and
\begin{equation}
\label{eq:sphericalharmonics}
    Y_l^m(\alpha,\beta) = \sqrt{\frac{2l+1}{4 \pi}\frac{(l-m)!}{(l+m)!}} P_l^m(\cos(\alpha)) e^{im \beta},
\end{equation}
where $P_l^m$ is the corresponding Legendre polynomial and $(\alpha,\beta)$ are the spherical coordinates on $\mathbb{S}^2$. 
{\color{black}
\begin{lemma}
    For every $l \in \N$, $m \in \{-l, \dots, l\}$, $Y_l^m$ and $Y_{l}^{-m}$ share the same modulus.   
\end{lemma}
}
\begin{proof}
    Recall that each $P_l^m$ satisfies 
\begin{equation*}
    P_l^{-m} = (-1)^m \frac{(l-m)!}{(l+m)!} P_l^m,
\end{equation*}
so that
\begin{equation}
\label{eq:relationspherical}
    Y_l^{-m}(\alpha,\beta) = (-1)^m e^{-2 i m \beta} Y_l^{m}(\alpha,\beta).
\end{equation}
Therefore, $Y_m^l$ and $Y_{-m}^l$ 
share the same modulus.
\end{proof}
{\color{black}
\begin{cor}
\label{prop:pointphaseeqS2}
For each $l\ge 1$, $H_0$ admits eigenfunctions sharing the same modulus inside the energy  level $l(l+1)$.  
\end{cor}
\begin{proof}
{\color{black}Take $l \geq 1$ and  $m \in \{1, \dots, l\}$. Then the eigenfunctions $Y_l^m$ and $Y_{l}^{-m}$ are linearly independent, share the same modulus, and correspond to the same energy level $l(l+1)$.}
\end{proof}
}
\begin{remark}
A natural open question is the following one: for $k,l \geq 0$, $k \neq l$, do there exist $\phi_k \in E_{k(k+1)}$ and $\phi_l \in E_{l(l+1)}$ sharing the same modulus?
\end{remark}

$3.$ Let us consider  the unit disk $\mathbb D = \{(x,y) \in \R^2\mid x^2 + y^2 \le  1\}$. 
Denote by $\Delta_{D}$
the Dirichlet--Laplace operator on $\mathbb D$. From \cite[Proposition 1.2.14]{Hen06}, the eigenvalues and eigenfunctions of $-\Delta_{D}$ are given, for every $k \geq 1$, by 
    \begin{align*}
        \lambda_{0,k} &= j_{0,k}^2,\qquad 
        u_{0,k}(r, \theta) = \sqrt{\frac{1}{\pi}} \frac{1}{|J_{0}'(j_{0,k})|} J_0(j_{0,k} r),
    \end{align*}
    and, for every $n,k \geq 1$,
    \begin{align*}
        \lambda_{n,k} &= j_{n,k}^2 \quad \text{of multiplicity 2},\\
        u_{n,k}(r, \theta) &= \sqrt{\frac{2}{\pi}} \frac{1}{|J_{n}'(j_{n,k})|} J_n(j_{n,k} r) \cos(n\theta),\\
          v_{n,k}(r, \theta) &= \sqrt{\frac{2}{\pi}} \frac{1}{|J_{n}'(j_{n,k})|} J_n(j_{n,k} r) \sin(n\theta),
    \end{align*}
where $j_{n,k}$ is the $k$-th zero of the Bessel function $J_n$ and $(r,\theta)$ are the polar coordinates on $\mathbb{D}$. 

\begin{prop}
\label{prop:eigendisk}
Let $n, m \geq 0$ and $k, l \geq 1$ be such that $(n,k) \neq (m,l)$. 
Assume that $\phi_{n,k}$ and $\phi_{m,l}$ are eigenfunctions corresponding to $j_{m,l}^2$ and $j_{n,k}^2$, respectively. 
Then $\phi_{n,k}$ and $\phi_{m,l}$ do not share the same modulus. 

On the other hand, for $n \geq 1$ and $k \geq 1$, there exist two $\C$-linearly independent eigenfunctions given by 
\begin{equation}
\label{eq:psinkdisk}
    \psi_{n,k}^+ = \sqrt{\frac{2}{\pi}} \frac{1}{|J_{n}'(j_{n,k})|} J_n(j_{n,k} r) e^{in \theta},\ \psi_{n,k}^-=\sqrt{\frac{2}{\pi}} \frac{1}{|J_{n}'(j_{n,k})|} J_n(j_{n,k} r) e^{-in \theta},
\end{equation}
corresponding to the eigenvalue $j_{n,k}^2$ that 
share the same modulus. 
\end{prop}

\begin{proof}
  In order to prove  the first part of the statement, 
  assume that $\phi_{n,k}$ and $\phi_{m,l}$ share the same modulus.
  Since $\phi_{n,k}$ is a nontrivial linear combination of the functions $u_{n,k}$ and $v_{n,k}$ introduced above, it follows that for almost all $\theta$ the function $(0,1)\ni r\to \phi_{n,k}(r,\theta)$ has $k-1$ zeros. Similarly, 
 $(0,1)\ni r\to \phi_{m,l}(r,\theta)$ has $l-1$ zeros for almost all $\theta$.
Since $\phi_{n,k}$ and $\phi_{m,l}$ have the same set of zeros, $k$ and $l$ should be equal. 
By the properties of the Bessel functions, moreover, $J_n(r)\simeq r^n$ and $J_m(r)\simeq r^m$ as $r$ tends to zero. We then deduce that $n=m$.

  The second part consists in remarking that the eigenfunctions given by \eqref{eq:psinkdisk} are $\C$-linearly independent.
  \end{proof}

$4.$ Let $M = \R^d$, $V(x) = \|x\|^2$, and $H_0 = - \Delta + V = - \Delta + \|x\|^2$ be the harmonic oscillator Hamiltonian. 

The eigenfunctions of the one-dimensional harmonic oscillator $H_0 = - \partial_{x}^2 + x^2$ are the Hermite eigenfunctions 
\begin{equation}
    \label{eq:functionhermited1}
    \Phi_k(x) = \frac{1}{\sqrt{2^k k! \sqrt{\pi}}} \left(x- \frac{d}{dx}\right)^k e^{-\frac{x^2}{2}},\qquad x \in \R,\quad k\in \N.
\end{equation}
The eigenvalue corresponding to $\Phi_k$ is $2k+1$ and $\Phi_k$ can be written as $\Phi_k(x) = H_k(x) e^{-\frac{x^2}{2}}$, 
where $H_k$ is a polynomial function of degree $k$. In the multi-dimensional case, for $d \geq 2$, $\alpha = (\alpha_j)_{1 \leq j \leq d} \in \N^d$, the Hermite eigenfunctions are given by
\begin{equation}
   \label{eq:functionhermitedgeneral}
   \Phi_{\alpha}(x) = \prod_{j=1}^d \Phi_{\alpha_j}(x_j)\qquad x = (x_j)_{1 \leq j \leq d} \in \R^d,\quad \alpha\in \N^d,
\end{equation}
and $\Phi_{\alpha}$ corresponds to the eigenvalue $2 |\alpha| + d$.

\begin{prop}
\label{prop:hermitedD}
Let $d=1$. For every $k_1,k_2 \in \N$, $k_1 \neq k_2$, $H_0$ does not admit two eigenfunctions corresponding 
to the energy levels $2 k_1 + 1$ and $2 k_2+1$ that share the same modulus. 
\end{prop}
\begin{proof}
    By contradiction, if $\Phi_k$ and $\Phi_\ell$ share the same modulus, then $|H_k| = |H_\ell|$, so that $H_k$ and $H_\ell$ have the same degree, which is not possible when $k\ne \ell$.
\end{proof}
Note that \Cref{prop:hermitedD} 
can  also be seen as
a consequence of \Cref{lem:simple} or \Cref{thm:1D}, see below.

\begin{remark}
     A natural open question is: 
     does the result of \Cref{prop:hermitedD} extend to the case $d \geq 2$? The difficulty is that in the contradiction argument, \rouge{we will have the equality in modulus for appropriate linear combinations of eigenfunctions, i.e.,
\begin{equation}
\left|\sum_{\alpha \in \N^d,\ |\alpha|=k_1} \gamma_{\alpha} \prod_{j=1}^d H_{\alpha_j}(x_j)\right| = \left|\sum_{\beta \in \N^d,\ |\beta|=k_2} \gamma_{\beta} \prod_{j=1}^d H_{\beta_j}(x_j)\right|\qquad x = (x_j)_{1 \leq j \leq d} \in \R^d.
\end{equation}
}
An argument based on the degree does not work anymore because a priori one can have cancellations in the previous equality, so that the polynomial on the left-hand side has the same degree as the polynomial on the right-hand side. 
\end{remark}

Eigenfunctions sharing the same modulus 
play a special role in
the bilinear control of quantum systems, as explained in the next section.

\subsection{Bilinear quantum  control systems}
\label{sec:bilinearquantumsystemsmanifolds}

Let $Q=(Q_1, \dots, Q_m)\in L^{\infty}_{\rm loc}(M;\R)^m$ be the potentials of interactions and let us consider the bilinear controlled Schr\"odinger equation
\begin{equation}
\label{eq:ScBilinear}
\left\{
\begin{array}{ll}
i \partial_t \psi = - \Delta_g \psi + V
\psi + \langle u(t), Q
\rangle_{\R^m} \psi& \text{in}\  (0,+\infty) \times M,\\
\psi = 0&\text{on}\  (0,+\infty) \times \cup_{j \in I_D}\mathcal{C}_j,\\
\partial_{\nu}\psi = 0&\text{on}\  (0,+\infty) \times \cup_{j \in I_N}\mathcal{C}_j,\\
\psi(0,\cdot) = \psi_0\quad& \text{in}\    M.
\end{array}
\right.
\end{equation}
In \eqref{eq:ScBilinear}, at time $t \in [0,+\infty)$, $\psi(t)\in L^2(M)$ is the state and $u(t) \in \R^m$ is the control.

In all the sequel, we also make the following hypothesis. 
\begin{ass}
\label{ass:qi}
   For every $j\in\{1,\dots,m\}$, $Q_j \in L^{\infty}(M;\R)$.
\end{ass}

By combining Assumptions~\ref{ass:op} and~\ref{ass:qi}, we have the following well-posedness result according to \cite{BMS82}. For every $T>0$, $\psi_0 \in L^2(M)$ and $u \in L^2(0,T;\R^d)$, there exists a unique mild solution $\psi \in C([0,T];L^2(M))$ of \eqref{eq:ScBilinear}, i.e.,
\begin{equation*}
    \psi(t) = e^{-it H_0} \psi_0 + \int_0^t e^{- i (t-s) H_0} \langle u(s), Q(x) \rangle \psi(s) ds,\qquad \forall 
    t \in [0,T].
\end{equation*}
If we assume furthermore that $\psi_0 \in \mathcal{S} = \{\psi \in L^2(M)\mid \|\psi\|_{L^2(M)} = 1\}$, then, for every $t \in [0,T]$, $\psi(t) \in \mathcal{S}$. In the sequel, for $\psi_0 \in L^2(M)$ and $u \in L^2(0,T;\R^d)$, the associated solution $\psi \in C([0,T];L^2(M))$ of \eqref{eq:ScBilinear} will be denoted by $\psi(\cdot; \psi_0, u)$.

For $\psi_0 \in L^2(M)$, let us denote
\begin{equation*}
    \mathcal{R}(\psi_0) := \{ \psi(t; \psi_0, u)\mid t \geq 0,\ u \in L^2(0,t;\R^d)\}.
\end{equation*}
From a controllability point of view, we have the following well-known negative result from \cite{Tur00}, based on \cite{BMS82} (see also \cite{BCC20} for extensions to the case of $L^1$ and impulsive controls): For every $\psi_0 \in \mathrm{Dom}(H_0) \cap \mathcal S$, the complement of $\mathcal{R}(\psi_0)$ in $\mathrm{Dom}(H_0) \cap \mathcal S$ is dense in $\mathrm{Dom}(H_0) \cap \mathcal S$, so in particular the interior of $\mathcal{R}(\psi_0)$ in $\mathrm{Dom}(H_0) \cap \mathcal S$ for the topology of $\mathrm{Dom}(H_0)$  is empty. With respect to this result of negative nature for the exact controllability, one may wonder if instead the approximate controllability holds. It is known that \eqref{eq:ScBilinear} is approximately controllable in large time, generically with respect to the parameters of the system, see \cite{MasonSigalotti2010}. Here, we will mainly focus on the small-time approximate controllability.
\begin{definition}
    An element $\psi_1 \in \mathcal{S}$ belongs to the small-time approximately reachable set from $\psi_0 \in \mathcal{S}$, denoted by $\overline{\mathcal{R}_0({\psi_0})}$, if for every $\varepsilon >0$, for every $\tau >0$, there exists $T \in (0,\tau]$ and $u \in L^2(0,T;\R^m)$ such that 
    \begin{equation*}
        \| \psi(T;\psi_0,u) - \psi_1\|_{L^2(M)} < \varepsilon.
    \end{equation*}
\end{definition}

The characterization of small-time approximately reachable sets for \eqref{eq:ScBilinear} is an open problem in general. For instance, 
{\color{black} it is shown in 
\cite{BCT18} that, for a class of systems of the type  \eqref{eq:ScBilinear} with $M = \R^d$ and for suitable
initial and final states, approximate controllability is possible but requires a positive minimal time.}
Results in the same spirit were also obtained in \cite{BBS21} by using a WKB method. Nevertheless, there are examples of 
bilinear systems for which $\overline{\mathcal{R}_0({\psi_0})}=\mathcal{S}$ for all $\psi_0\in \mathcal{S}$, see \cite{BCC12}. 

In this article, we mainly focus on a weaker notion, that we call small-time 
isomodulus
approximate controllability.
\begin{definition}
We say that \eqref{eq:ScBilinear} is 
\emph{small-time 
isomodulus 
approximately controllable}
from $\psi_0 \in \mathcal{S}$ if
    \begin{equation*}
        \{  e^{i \theta} \psi_0\mid \theta \in L^2(M;\T)\} \subset \overline{\mathcal{R}_0({\psi_0})}.
    \end{equation*}
\end{definition}

Let us investigate for particular examples of Riemannian manifolds $M$, 
operators $H_0 = - \Delta_g + V$, 
and potentials of interaction $Q_1, \dots, Q_m$ when small-time 
isomodulus
approximate controllability holds.\\

$1.$ In the case where $M$ is a torus the following result holds.

\begin{theorem}{\cite[Theorem A]{DN21}}\label{th:DN21}
Let $M = \T^d$, let $\mathcal K$ be the set 
defined by
\begin{equation*}
    \mathcal K = \{(1, 0, \dots, 0), (0,1,0 \dots, 0), \dots, (0, \dots, 0, 1, 0), (1, 1, \dots, 1)\}\subset \R^d,
\end{equation*}
and assume that 
$Q_1,\dots,Q_m$ are
such that
\begin{equation*}
    x \mapsto 1,\ x\mapsto \sin\langle x, k\rangle,\ x\mapsto \cos\langle x, k\rangle\ \in \mathrm{span} \{Q_1, \dots, Q_m\},\qquad \forall k \in \mathcal K.
\end{equation*}
Then
    \eqref{eq:ScBilinear} is small-time 
   isomodulus
   approximately controllable.
\end{theorem}
Note that \cite[Theorem A]{DN21} actually claims the small-time isomodulus
approximate controllability for $\psi_0 \in H^s(M) \cap \mathcal S$ for $s > d/2$ but this is due to the eventual presence of the semi-linearity in their equation.

A striking consequence of such a result is that the small-time approximate controllability among particular eigenstates holds.
\begin{cor}
{\cite[Theorem B]{DN21}}
\label{thm:chambrionpozollispheresmalltime}
 Take $V=0$. Then, we have
\begin{equation*}
    \Phi_k^{\pm} \in \overline{\mathcal{R}_0(\Phi_\ell^{\pm})},
\end{equation*}
where  $\Phi_k^{\pm},\Phi_\ell^{\pm}$ are any  eigenfunctions defined as in \eqref{eq:eigenfunctiontorus}.
\end{cor}
In particular, the system can be steered arbitrarily fast from one energy level to any other one. 
This is due to the fact that for every $k, \ell 
$, $\Phi_k^\pm$ and $\Phi_\ell^\pm$ 
share the same modulus.\\

$2.$ For the two-dimensional sphere the following result holds.
\begin{theorem}{\cite[Theorem 3]{CP22}}
    Let $M=\mathbb S^2$, $V=0$, $Q_1(x,y,z)=x$, $Q_2(x,y,z)=y$, $Q_3(x,y,z)=z$. Then \eqref{eq:ScBilinear} is small-time isomodulus
    approximately controllable.
\end{theorem}
As an application, we have the following result, that claims the small-time approximate controllability between particular spherical harmonics.
\begin{cor}{\cite[Theorem 3]{CP22}}
We have
\begin{equation}
\label{eq:chambrionpozollieigen}
    Y_m^l \in \overline{\mathcal{R}_0(Y_{-m}^l)},\qquad \forall l \in \N,\ \forall m \in \{-l, \dots, l\},
\end{equation}
where $Y_m^l$ are the eigenfunctions defined in \eqref{eq:sphericalharmonics}.
\end{cor}
Note that \cite[Theorem 3]{CP22} actually claims \eqref{eq:chambrionpozollieigen} for $l \in \N$ and $m=\pm l$ but the proof is still valid for any $m \in \{-l, \dots, l\}$ because of \Cref{prop:pointphaseeqS2}. \\

One of the  contributions of this article is the generalization of \cite[Theorem A]{DN21} and \cite[Theorem 3]{CP22} to Riemannian manifolds with boundaries, assuming that the potentials of interaction satisfy a saturation property that we describe now. 

We assume that $Q_1, \dots, Q_m \in C^{\infty}(M;\R)$ and let us define the sequence of positive cones
\begin{equation}
\label{eq:defH0subspace}
    \mathcal{H}_0 = \left\{\varphi \in \text{span}\{Q_1, \dots, Q_m\} \mid \varphi \mathrm{Dom}(H_0)  \subset \mathrm{Dom}(H_0) \right\} \subset L^2(M;\R),
\end{equation}
and iteratively for $N \geq 0$,
\begin{multline}
\label{eq:defHNsubspacerecurrence}
\mathcal H_{N+1} = \left\{ \varphi \in \mathcal{H}_N +\left\{- \alpha g(\nabla_g \psi,\nabla_g \psi) \mid \pm\psi \in \mathcal{H}_N,\ \alpha \geq 0\right\}\mid \varphi \mathrm{Dom}(H_0)  \subset \mathrm{Dom}(H_0) \right\}.
\end{multline}
Finally, we define 
\begin{equation}
\label{eq:defHinfty}
    \mathcal{H}_{\infty} = \bigcup_{N \geq 0} \mathcal{H}_N.
\end{equation}
\begin{theorem}
\label{th:smalltimepointphase}
    Assume that $\mathcal H_{\infty}$ is 
    dense in $L^2(M;\R)$. Then, \eqref{eq:ScBilinear} is small-time isomodulus approximately controllable.
\end{theorem}
It is worth mentioning that the proof of \cite[Theorem A]{DN21} directly gives \Cref{th:smalltimepointphase} when $M=\T^d$ and the proof of \cite[Theorem 3]{CP22} directly gives \Cref{th:smalltimepointphase} when $M$ is a manifold without boundary. Even if the treatment of the boundary 
 is not the most
difficult issue, 
we give a full proof of \Cref{th:smalltimepointphase} in Appendix~\ref{sec:appendixsmalltimepointphase} in an abstract setting for the sake of completeness and for the application to other quantum control systems, in particular on quantum graphs.

As an application of \Cref{th:smalltimepointphase}, one can then consider the case of the Dirichlet--Laplace operator on the disk.
\begin{cor}
\label{cor:transitiondisk}
    Assume that $Q_1, \dots, Q_m \in C^{\infty}(\mathbb D;\R)$ and $\mathcal{H}_{\infty}$ is dense in $L^2(\mathbb D;\R)$. Then, 
    \begin{equation*}
        \psi_{n,k}^+ \in \overline{\mathcal{R}_0(\psi_{n,k}^-)},\qquad \forall n,k \geq 1,
    \end{equation*}
    where the eigenfunctions $\psi_{n,k}^{\pm}$ are defined in \eqref{eq:psinkdisk}.
\end{cor}
Here, the $C^{\infty}$-regularity up to the boundary of the potentials  leads to the fact that one can remove the conditions of the stabilization of $\text{Dom}(H_0)$ by the functions $\varphi$ in the definitions \eqref{eq:defH0subspace} and \eqref{eq:defHNsubspacerecurrence} because it is automatically satisfied. This is specific to the homogeneous Dirichlet boundary conditions. It is not the case for homogeneous Neumann boundary conditions.

\section{Eigenfunctions sharing the same modulus}
\label{s-moduli}

\subsection{General properties}

The goal of this part is to study the implications of having two 
linearly independent 
eigenfunctions 
sharing the same modulus.

Let $H_0$ be as in Assumption~\ref{ass:op}.
Notice that we could relax the assumptions on $H_0$ in this section, not requiring its entire spectrum to be discrete, but just focusing on 
 two 
  of its eigenvalues
  $\lambda_k$ and $\lambda_\ell$ 
  with corresponding eigenfunctions $\phi_k,\phi_\ell:M\to \C$. 
 By elliptic regularity and Sobolev embeddings, see for instance \cite[Theorem 7.26, Theorem 9.11]{GT01}, 
 for every $0 < \alpha < 1$ one has that
$\phi_k$ and $\phi_\ell$ are $C^{1,\alpha}_{\rm loc}(M;\C)$.
  The one-dimensional case is specific because, by standard ODE arguments, one can prove that $\phi_k$ and $\phi_\ell$ are $C^{1,1}_{\rm loc}$.

  Let us assume that $\phi_k$ and $\phi_\ell$
  share the same modulus 
  $\rho:=|\phi_k|=|\phi_\ell|\in C(M,[0,+\infty))$.
Set $M_\rho=\{x\in M\mid \rho(x)\ne 0\}$ and  let
$\theta_k,\theta_\ell:M_\rho\to \mathbb{T}
=\R/2\pi\Z$ 
be
such that 
\begin{equation}\label{k-ell-connected}
\phi_k(x) = \rho(x) e^{i \theta_k(x)}, \quad \phi_\ell(x) = \rho(x) e^{i \theta_\ell(x)},\qquad x\in M_\rho.
\end{equation}
Notice that both $\phi_k$ and $\phi_\ell$ vanish on $M\setminus M_\rho=\{x\in M\mid \rho(x)=0\}$ and 
that, by unique continuation, see \cite{Aro57}, $M_\rho$ is dense in $M$.
Moreover, 
 for every $0 < \alpha < 1$ one has that
 $\rho|_{M_\rho}$ and $\theta_k:M_\rho\to \mathbb{T}
$ are $C^{1,\alpha}_{\rm loc}$, and even  $C^{1,1}_{\rm loc}$ in the one-dimensional case.

Let us observe the following.

\begin{lemma}\label{lem:simple}
    If $\lambda_k$ and $\lambda_\ell$ are simple and distinct eigenvalues of the Schr\"odinger operator $H_0$ with corresponding eigenfunctions $\phi_k$ and $\phi_\ell$, 
    then 
$\phi_k$ and $\phi_\ell$ cannot 
share the same modulus.
\end{lemma}
\begin{proof}
Assume by contradiction that $\phi_k$ and $\phi_\ell$ 
share the same modulus
and let $\rho,\phi_k,\phi_\ell$ be as in 
\eqref{k-ell-connected}. 

Notice that, since $V$ is real-valued,  $\RE(\phi_k)$ and $\IM(\phi_k)$ satisfy
$$- \Delta \RE(\phi_k) + V \RE(\phi_k) = \lambda_k \RE(\phi_k),\qquad  - \Delta \IM(\phi_k) + V \IM(\phi_k) = \lambda_k \IM(\phi_k).$$
By using that $\lambda_k$ is a simple eigenvalue, we then deduce that $\RE(\phi_k)$ and $\IM(\phi_k)$ are colinear, that is, 
up to exchanging the roles of $\RE(\phi_k)$ and $\IM(\phi_k)$,
there exists $\mu_k \in \R$ such that
$\RE(\phi_k) = \mu_k \IM(\phi_k)$, i.e., $\rho \cos(\theta_k) = \mu_k \rho \sin (\theta_k)$.
We therefore deduce that
$$ \cos(\theta_k) = \mu_k  \sin (\theta_k)\qquad \text{on}\ M_{\rho}.$$
By using that $M_{\rho}$ is 
open and $\theta_k$ is continuous on $M_\rho$, we then deduce that $\theta_k$ is constant on each connected component $M_{\rho}$. As a consequence, we obtain that
$$ - \Delta \rho + V \rho = \lambda_k \rho\qquad \text{on}\ M_{\rho}.$$
By using the same argument for $\lambda_\ell$, 
we obtain
$$ - \Delta \rho + V \rho = \lambda_\ell \rho\qquad \text{on}\ M_{\rho}.$$
Hence, since $M_\rho$ is nonempty, 
$ \lambda_k = \lambda_\ell$, contradicting the hypotheses.
\end{proof}
\begin{remark}
    One may wonder if the conclusion of \Cref{lem:simple} still holds true assuming that only 
    $\lambda_k$ is simple. This turns to be false by considering the example of $M = \T$ and $H_0 = -\Delta$. Indeed, $0$ is a simple eigenvalue with corresponding  eigenfunction constantly equal to $1$, while $1$ is an eigenvalue of multiplicity $2$ with corresponding eigenfunctions $e^{ix}$ and $e^{-ix}$, and   the three eigenfunctions obviously share the same modulus.
\end{remark}

\begin{cor}\label{cor:gen-metric}
Let $M$ be a compact connected $C^{\infty}$ manifold without boundary of dimension larger than or equal to $2$. Then, generically with respect to the Riemanniann metric, 
no pair of $\C$-linearly independent eigenfunctions
of the Laplace--Beltrami operator $-\Delta_g$ 
share the same modulus.
\end{cor}
\begin{proof}
    This is straightforward application of Lemma~\ref{lem:simple} and the well-known result by Uhlenbeck~\cite{Uhlenbeck1976,Tanikawa1979}
 \footnote{The version of Uhlenbeck's result given by Tanikawa in \cite{Tanikawa1979}
 provides the formulation in the $C^\infty$ topology and
 also  covers the case where $M$ has nonempty boundary and the Schr\"odinger operator has Dirichlet boundary conditions. One could therefore extend the statement of Corollary~\ref{cor:gen-metric} accordingly.} 
 ensuring that, given a compact connected $C^{\infty}$ manifold $M$ of dimension larger than or equal to $2$,  the set
    $$ 
    \{g \in {\cal M}\mid \text{all eigenvalues of } -\Delta_g\ \text{have multiplicity one}\},$$
is residual in the topological space ${\cal M}$, \rouge{where $\cal M$ is the set of all smooth Riemannian metrics on $M$ endowed with the $C^\infty$ topology.}
\end{proof}

\begin{remark}
 Results similar to Corollary~\ref{cor:gen-metric} 
can be obtained when considering the genericity with respect to the potential $V$ (with no restriction on the dimension of $M$). Indeed, the spectrum of the Schr\"odinger operator is 
known to be
generically simple with respect to $V$. This general statement requires precise functional analysis setups to be  rigorously formulated. 
For the case of compact connected manifolds without boundary and $C^\infty$ potentials $V$ see \cite{Albert1975}.
The case of bounded Euclidean domains with Dirichlet boundary conditions and $V\in L^\infty$ or of full Euclidean spaces with $V\in L^\infty_{\rm loc}$ such that $\lim_{\|x\|\to\infty}V(x)=+\infty$ can be found in \cite[Proposition 3.2]{MasonSigalotti2010}.
\end{remark}

\begin{lemma}
If $\mathrm{dim}(E_{\lambda_k}) \geq 2$, then $H_0$ admits eigenfunctions sharing the same modulus inside the energy level $\lambda_k$.
\end{lemma}
\begin{proof}
Let $\phi_{k_1}$ and $\phi_{k_2}$ be two linearly independent \rouge{real-valued} eigenfunctions corresponding to the same eigenvalue $\lambda_k$. Then $
 \phi_{k_1} + i \phi_{k_2}$ and $\phi_{k_1} - i \phi_{k_1}$ are two linearly independent eigenfunctions sharing the same modulus inside the energy level $\lambda_k$.
\end{proof}


Let us collect in the following result some differential identities satisfied by the modulus and phases of two eigenfunctions sharing the same modulus.

\begin{lemma}
    Let  $\phi_k,\phi_\ell:M\to \C$
be two eigenfunctions of the Schr\"odinger operator $A$ sharing the same modulus with corresponding eigenvalues $\lambda_k,\lambda_\ell$. 
Let $\rho$, $\theta_k$, $\theta_\ell$ be as in \eqref{k-ell-connected}. 
The following elliptic equations are fulfilled
\begin{align}\label{Lap-rho}
- \Delta \rho + |\nabla \theta_j|^2  \rho + V \rho &= \lambda_j \rho\qquad \mbox{on } M_\rho,\qquad 
j=k,l,
\\
\label{Lap-theta}
- \rho \Delta \theta_j - 2 \nabla \rho \cdot \nabla \theta_j &= 0\qquad \mbox{on } M_\rho,\qquad
j=k,l.
\end{align}
Moreover, 
\begin{equation}\label{eq:thetas}
|\nabla \theta_k|^2 - |\nabla \theta_\ell|^2 = \lambda_k - \lambda_\ell\qquad \mbox{on } M_\rho.
\end{equation}
\end{lemma}
\begin{proof}
    First, we have
    \begin{equation}\label{eq-firstwehave}
\Delta \phi_k= \Delta \rho e^{i \theta_k} + 2 i \nabla \rho \cdot \nabla \theta_k e^{i \theta_k} + \rho (i \Delta \theta_k e^{i \theta_k} - |\nabla \theta_k|^2 e^{i \theta_k})\qquad \mbox{on }M_\rho.
\end{equation}
Plugging \eqref{eq-firstwehave} in the equality
$-\Delta \phi_k + V \phi_k = \lambda_k \phi_k$,
simplifying by $e^{i \theta_k}$, and taking the real and imaginary parts of the obtained equality, we get \eqref{Lap-rho} and \eqref{Lap-theta}. 
Equation \eqref{eq:thetas} is obtained by taking the difference of the two equations in \eqref{Lap-rho} corresponding to $j=k,l$ and by simplifying by $\rho \neq 0$ on $M_{\rho}$.
\end{proof}

\subsection{The one-dimensional case}
\label{ss-1}

A noticeable feature of  the one-dimensional case is that, 
locally around each point of $M$, 
the Riemannian manifold is isometric to an  interval in $\R$ endowed with the Euclidean metric, as it follows by taking an arclength coordinate (see, for instance, \cite[Appendix Classifying $1$-dimensional manifolds]{Mil65}). Hence, since $M$ is complete, four situations may occur: 
\begin{itemize}
    \item 
$M$ is isometric to the line $\R$, 
\item $M$ is isometric to the half-line $[0,+\infty)$, 
\item $M$ is isometric to a compact  interval $[0,L]$ for some $L>0$, 
\item $M$ is a closed curve isometric to the quotient $\R/L\Z$ for some $L>0$.
\end{itemize}

Moreover, 
in the arclength-coordinate the metric is constant and
the Laplace--Beltrami operator coincides with the second derivative.

The main result of this section is Theorem~\ref{thm:1D} below, stating, in particular, that a one-dimensional manifold $M$ admits a 
pair of
eigenfunctions sharing the same modulus and
corresponding to different eigenvalues
only if $M=\mathbb{T}$  and $V$ is constant, that is, in the one-dimensional occurrence of the system studied in \cite{DN21}. 

\begin{theorem}\label{thm:1D}
If $M$ is one-dimensional and the Schr\"odinger operator 
$H_0$ admits two 
$\C$-linearly independent
eigenfunctions $\phi_k$ and $\phi_\ell$
sharing the same modulus, 
then necessarily $M$ is 
a closed curve and $\phi_k$, $\phi_\ell$ are nowhere vanishing on $M$. If, moreover, the two eigenfunctions correspond to distinct eigenvalues, then $V$ is constant. 
\end{theorem}
The proof of the previous result is based on the following lemma.
\begin{lemma}\label{lem:inab}
Let $(a,b)$ be an interval in $\R$ and $\lambda_k,\lambda_\ell\in \R$.
    Assume that 
    $\rho\in C^1(a,b;(0,+\infty))$ and $\theta_k,\theta_\ell\in C^{1,1}(a,b;\mathbb{T})$
are 
such that
    $$ \rho \neq 0, \quad- \rho \theta_k'' - 2 \rho' \theta_k' = 0, \quad 
  - \rho \theta_\ell'' - 2 \rho' \theta_\ell' = 0,\quad
    | \theta_k'|^2 - |\theta_\ell'|^2 = \lambda_k - \lambda_\ell\qquad
    \ \text{in}\ (a,b).
    $$
    Then there exist $C_k,C_\ell\in \R$ such that $\theta_k'=C_k\rho^{-2}$ and $\theta_\ell'=C_\ell\rho^{-2}$ in $(a,b)$. 
   Moreover,   if $\lambda_k= \lambda_\ell$, then there exists $c\in \T$ such that 
     either
     $\theta_k -\theta_\ell=c$  or $\theta_k+\theta_\ell=c$ 
     in $(a,b)$.
If, instead, $\lambda_k\ne \lambda_\ell$ then 
$C_k^2\ne C_\ell^2$ and
there exists $\rho_*\in (0,+\infty)$ 
such that
 \begin{equation}\label{eq:affine}
 \rho(x) = \rho_*,\quad 
 \qquad   x\in (a,b).
 \end{equation}
\end{lemma}

\begin{proof}
Setting $\gamma_k = \theta_k'$, we have that 
    \begin{equation*}
\gamma_k' = - 2 (\rho'/ \rho) \gamma_k\qquad 
\text{almost everywhere in}\ (a,b).
\end{equation*}
Integrating, we deduce that there exists $C_k\in \R$ such that
\begin{equation}\label{eq:Crho-2}
    \theta_k'(x)= \gamma_k(x) = C_k \rho(x)^{-2}, \qquad x\in (a,b).
     \end{equation}
Analogously, there exists $C_\ell\in\R$ such that $\theta_\ell'(x) = C_\ell \rho(x)^{-2}$, $x\in (a,b)$. Hence, using the equality 
$| \theta_k'|^2 - |\theta_\ell'|^2 = \lambda_k - \lambda_\ell$, we deduce that 
\begin{equation}\label{eq:exf}
    (C_k^2-C_\ell^2)\rho(x)^{-4}=\lambda_k - \lambda_\ell,\qquad x\in (a,b).
\end{equation}    
If $\lambda_k=\lambda_\ell$ then necessarily $C_k^2=C_\ell^2$, meaning that either $\theta_k'=\theta_\ell'$ or $\theta_k'=-\theta_\ell'$ in $(a,b)$. The latter means that either  $\theta_k-\theta_\ell$ or $\theta_k+\theta_\ell$ is constant. 

Let now consider the case $\lambda_k\ne\lambda_\ell$.
We deduce from \eqref{eq:exf} that 
$C_k^2\ne C_\ell^2$ and that
 $\rho$ is constant on $(a,b)$, that is, there exists $\rho_*>0$ such that
$ \rho(x) =\rho_*$, $x\in (a,b)$. 
\end{proof}

\begin{proof}[Proof of Theorem~\ref{thm:1D}]
Let $\phi_k$ and $\phi_\ell$ be two 
eigenfunctions
of $H_0$ sharing the same modulus and corresponding to the eigenvalues $\lambda_k$ and $\lambda_\ell$. 
In particular, $\phi_k$ and $\phi_\ell$ are assumed to be $\C$-linearly independent. 
Let $\rho,\theta_k,\theta_\ell$ be as in \eqref{k-ell-connected}. 
Up to taking an arclength coordinate, 
equations~\eqref{Lap-theta} and \eqref{eq:thetas}
read $-\rho \theta_k''-2 \theta_k'\rho=0=-\rho \theta_\ell''-2 \theta_\ell'\rho$ 
and $|\theta_k'|^2-|\theta_\ell'|^2=\lambda_k-\lambda_\ell$
on $M_\rho$. These are the equations for $\theta_k,\theta_\ell,\rho$ appearing in the statement of Lemma~\ref{lem:inab}.

Let us first 
exclude the case in which 
$H_0$ has boundary conditions of Neumann type  at a point of $\partial M$. 
Without loss of generality 
such point is $0$ and $M=[0,L]$ for some $L>0$ or $M=[0,+\infty)$. 
 Since for every $\lambda
 \in \R$ 
the space of real-valued solutions $\varphi$ of 
$-\varphi''(x)=(\lambda-V(x)) \varphi(x)$ with
$
\varphi'(0)=0$  is one-dimensional, then  $\lambda_k$ and $\lambda_\ell$ are both simple and the conclusion follows from Lemma~\ref{lem:simple}.

Let us consider now the case $\lambda_k\ne \lambda_\ell$.
We obtain from Lemma~\ref{lem:inab} that $\rho$ is locally constant on $M_\rho$. Since, moreover, $M$ is connected and $\rho$ is continuous,  then  $\rho$ coincides with some constant $\rho_*>0$ on $M$. 
This excludes the case in which 
$H_0$ has boundary conditions of Dirichlet type at a point of $\partial M$ and the case in which 
$M$ is isometric to $\R$ or to $[0,+\infty)$, since otherwise $\phi_k$ and $\phi_\ell$ would not be in $L^2$. In particular $M$ is a closed curve. 
We also obtain 
 from \eqref{Lap-rho} and \eqref{eq:affine}
 that $V$ is constant (globally on $M=M_\rho$).

Let us now turn to the case where $\lambda_k=\lambda_\ell$. 
We claim that $\rho$ does not vanish on $M$. 
Assume by contradiction that this is not the case and 
identify isometrically a connected component of $M_\rho$ with a possibly unbounded interval
$(a,b)$ or $(a,b]$, with  $a\in M\setminus M_\rho$.
 Since $\phi_k'e^{-i \theta_k}=\rho'+i\theta_k'\rho$ on $(a,b)$ and both $\rho'$ and $\theta_k'\rho$ are real-valued,  we deduce that 
$\rho'$ and $\theta_k'\rho$ are bounded in a right-neighborhood of $a$. 
By Lemma~\ref{lem:inab}, there exists $C_k\in \R$ such that $\theta_k'\rho=C_k\rho^{-1}$ on $(a,b)$. Since $\rho(x)\to 0$ as $x\downarrow a$,
it follows from the boundedness of $\theta_k'\rho$ that $C_k=0$. Similarly, $C_\ell=0$, 
yielding that $\theta_k$ and $\theta_\ell$ are constant on a connected component of $M_\rho$. 
Hence $\phi_k$ and $\phi_\ell$ are $\C$-linearly dependent on a nonempty open set of $M$, yielding by the unique continuation principle that they are $\C$-linearly dependent on $M$. This concludes the contradiction argument proving that $\rho$ does not vanish on $M$ (excluding, in particular, the case in which 
$H_0$ has boundary conditions of Dirichlet type at a point of $\partial M$).

We are then left to exclude the case where $M=M_\rho$ is isometric to $\R$.
It follows from Lemma~\ref{lem:inab} that 
either $\theta_k-\theta_\ell$ or $\theta_k+\theta_\ell$ is constant  on $M$. By $\C$-linear independence of $\phi_k$ and $\phi_\ell$ we can exclude the case where 
$\theta_k-\theta_\ell$ is constant and we can assume that $\theta_k'=C_k \rho^{-2}$ on $M$ with $C_k\ne 0$.

Since $\phi_k$ and $\phi_\ell$ are in $H^1(M)$, then 
\begin{equation}\label{eq:liminrho}
    \lim_{|x|\to \infty}\rho(x)= 0.
\end{equation} 

It follows from 
\eqref{Lap-rho} that 
$\rho$ satisfies
\[-\rho''+C_k^2\rho^{-3}+V\rho=\lambda_k \rho\qquad\mbox{on $M=\R$}.\]

Since 
$V$ is bounded from below, we deduce that 
there exists $\rho_*>0$ such that
$\rho''(x)>1$ at every $x$ such that $\rho(x)<\rho_*$. 
This is in contradiction with \eqref{eq:liminrho} and the proof is concluded. 
\end{proof}

\begin{remark}\label{rem:exp}
The theorem allows to conclude that $V$ is constant only in the case where $\lambda_k\ne \lambda_\ell$. Let us see through an example that this is not true in general if $\lambda_k= \lambda_\ell$.
We proceed in a constructive way, 
starting from any non-constant $\rho\in C^{1,1}(\T;(0,+\infty))$ and 
selecting $V$ and $\theta$, both non-constant, such that 
$\rho e^{i\theta}$ is an eigenfunction of $-\Delta+V$ corresponding to the eigenvalue $0$. 
In particular, $\rho e^{i \theta}$ and $\rho e^{-i \theta}$
are $\C$-linearly independent eigenfunctions of $-\Delta+V$ on $\T$.

Consider a positive integer $j$ to be fixed later and set 
\begin{align*}
C&=\frac{2\pi}{\int_{0}^{2\pi}\rho(x)^{-2}\,dx},\\
    \theta(x)&=C j\int_0^x\rho(y)^{-2}\,dy,\qquad     x\in [0,2\pi].
\end{align*}
It holds by construction that $\theta(0)=0$ and 
$\theta(2\pi)=2j\pi$. Hence $\theta$, seen as a function from $\T$ to itself,  is well defined and non-constant.
Notice that $\theta$ has been defined in such a way that $\theta'=Cj\rho^{-2}$. 
Setting then 
\[V=\frac{\rho''}{\rho}-(\theta')^2=\frac{\rho''}{\rho}-\frac{C^2j^2}{\rho^4}\qquad \mbox{on }\T,\]
we have that $\rho$ solves 
\[-\rho''+(\theta')^2 \rho+V\rho=0,\qquad \mbox{on }\T.\]
Moreover, $j$ can be chosen so that $V$ is non-constant. 

One can now check that $\phi=\rho e^{i \theta}$ solves 
$-\phi''+V \phi=0$. As a consequence, $\bar \phi=\rho e^{-i \theta}$ solves the same equation. This completes our construction. 

For an explicit example, one can take  $j=1$ and 
\[\rho(x)=\cos(x)+2,\qquad x\in \T=\R/2\pi\Z.\]
The resulting $V$ and the linearly independent eigenfunctions $\rho\cos(\theta)$ and $\rho\sin(\theta)$   of $-\Delta+V$ are illustrated in Figure~\ref{fig-exp}.
\begin{figure}[h]
    \centering
    \includegraphics[width=5cm]{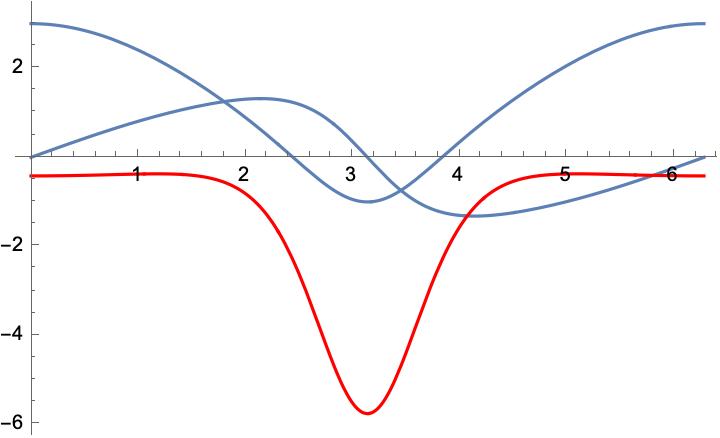}
    \caption{Plot of $V$ (in red), $\rho\cos(\theta)$ and $\rho\sin(\theta)$ (in blue) from Remark~\ref{rem:exp}.\label{fig-exp}}
\end{figure}

As a consequence of Theorem~\ref{th:DN21} in the case $d=1$, even in the case $V\ne 0$ 
it might be possible to
steer system~\eqref{eq:ScBilinear} in arbitrarily small time from one eigenfunction arbitrarily close to a linearly independent one (sharing the same eigenvalue). 
Notice that the Schr\"odinger equation on $\T$ with non-constant potential $V$ 
has a relevant role in solid state physics (see, e.g, \cite[Section~XIII.16]{ReedSimon}). 
\end{remark}

\section{Quantum graphs}
\label{s-grafi}
\subsection{Setting}

Let $\mathcal G$ be a compact connected metric
graph with $N$ edges $e_1,\dots,e_N$ of respective lengths $L_1,\dots,L_N$ and 
$M$ vertices $v_1,\dots,v_M$. For every vertex $v$, we denote 
\begin{align}\label{molteplicit}N(v):=\big\{l \in\{1,\dots,N\}\ |\ v\in 
e_l\big\},\qquad n(v):=|N(v)|.
\end{align}

We denote by $\mathcal{V}_{i}$ the set of internal vertices of $\mathcal G$, that is, the set of those vertices $v$ for which either $n(v)>1$ or $n(v)=1$ and the single edge to which $v$ belongs is actually a loop. The remaining vertices are called external and we denote by $\mathcal{V}_{e}$ the set comprising them.

Each edge $e_j$ is parameterized by an arclength coordinate going from $0$ to its 
length $L_j$.
On $\mathcal G$ 
we 
consider functions $f:=(f^1,\dots,f^N)$,
where $f^j:e_j\rightarrow \C$ for every $j=1, \dots, N$. We denote
$$L^2(\mathcal G)=\bigoplus_{j=1}^N L^2(e_j).$$
The Hilbert space $L^2(\mathcal G)$ is  equipped 
with the norm $\|\cdot\|_{L^2}$ defined by
\begin{equation*}
    \norme{f}_{L^2(\mathcal G)}^2 = \sum_{j=1}^N \norme{f^j}_{L^2(e_j)}^2.
\end{equation*}
We then denote by  $H^1(\mathcal G)$ the Sobolev space of all  functions $f=(f^1,\dots,f^N)$ on $\mathcal G$ that are continuous (that is, such that $f^{j}(v)=f^l(v)$ if $j,l\in N(v)$) and such that $f^j$ 
belong to $H^1(e_j)$ for every $j=1,\dots,N$, equipped with the norm defined by
\begin{equation*}
    \norme{f}_{H^1(\mathcal G)}^2 = \sum_{j=1}^N \norme{f^j}_{H^1(e_j)}^2.
\end{equation*}
Finally, the Sobolev space $\tilde{H}^2(\mathcal G)$ is defined as
$$\tilde{H}^2(\mathcal G)=\bigoplus_{j=1}^N H^2(e_j),$$
equipped with the natural norm.

Let $f\in \tilde{H}^2(\mathcal G) \cap H^1(\mathcal G)$ and $v$ be a vertex of $\mathcal G$ connected once to an edge $e_j$ with 
$j=1,\dots,N$. When the coordinate parameter $e_j$ at the vertex $v$ is equal to $0$ (respectively, $L_j$), we denote
\begin{equation}
\label{eq:derivativeatavertex}
\partial_x f^j(v)=\partial_xf^j(0),\qquad \big(\text{respectively,}\ \partial_x 
f^j(v)=-\partial_xf^j(L_j)\big).
\end{equation}
When $e_j$ is a loop connected to $v$ at both its extremities, we use the notation 
\begin{equation}
\label{eq:derivativeatavertexloop}
\partial_x 
f^j(v)=\partial_xf^j(0)-\partial_xf^j(L_j).
\end{equation}

In the following, we consider two type of boundary conditions. For $f \in \tilde{H}^2(\mathcal G) \cap H^1(\mathcal G)$, we say that $f$ satisfies 
\begin{itemize}
\item \emph{a Dirichlet boundary condition} at a vertex $v\in \mathcal{V}_{e}$ if $f(v) = 0$,
\item \emph{a Neumann--Kirchoff boundary condition} at a vertex $v \in \mathcal{V}_{e}\cup \mathcal{V}_i
$ if $\sum_{j\in N(v)}\partial_x f^j(v)=0$.
\end{itemize}
We partition $\mathcal{V}_e$ in a two subsets of vertices $\mathcal{V}_D$ and $\mathcal{V}_N$ at which we impose, respectively, 
a Dirichlet and a Neumann--Kirchoff 
boundary condition.

Let $V \in L^{\infty}(\mathcal G; \R)$, $H_0 = - \partial_{x}^2 + V$ be the Hamiltonian, and let us define
\begin{align*}
    \mathrm{Dom}(H_0) = \{f \in \tilde{H}^2&(\mathcal G) \cap H^1(\mathcal G)\mid \\
    &
    f(v) = 0\ \forall v \in \mathcal{V}_D \mbox{ and }
    \sum_{j\in N(v)}\partial_x f^j(v)=0\ \forall v \in \mathcal{V}_N \cup \mathcal{V}_i\}.
\end{align*}

\begin{definition}
A quantum graph is a compact connected metric graph $\mathcal G$, equipped with the unbounded operator $(H_0, \mathrm{Dom}(H_0))$. 
\end{definition}

Given a quantum graph,  the Hamiltonian $H_0$ is a self-adjoint operator that has compact resolvent, see \cite[Theorems~1.4.4 and 3.1.1]{BK13}. So, there exists an orthonormal basis of eigenfunctions $(\phi_k)_{k \geq 1}$ associated with the sequence of real eigenvalues $(\lambda_k)_{k \geq 1}$.

\subsection{General results 
for quantum graphs}

The definition of 
eigenfunctions sharing the same modulus
is the same as in \Cref{def:pointphase}.

\begin{lemma}\label{lem:simplegraph}
    If $\lambda_k$ and $\lambda_\ell$ are simple and distinct eigenvalues of the Hamiltonian $H_0= - \partial_{x}^2 + V$ with corresponding eigenfunctions $\phi_k$ and $\phi_\ell$, 
    then 
$\phi_k$ and $\phi_\ell$ cannot 
share the same modulus.
\end{lemma}
\begin{proof}
    The proof is analogue to the one of \Cref{lem:simple}.
\end{proof}

\begin{cor}
\label{cor:genericpointphaseequivalentgraph}
Let $\mathcal G$ be a compact graph 
with the Hamiltonian $H_0 = - \partial_{x}^2$, equipped with Neumann--Kirchoff conditions. 
Assume that $\mathcal G$ has at least one vertex 
$v$ with $N(v)>1$.
Then, generically with respect to the lengths of the edges of the graph, no pair of distinct eigenfunctions
belonging to the same orthonormal basis 
of the Laplace--Beltrami operator 
share the same modulus.
\end{cor}

\begin{proof}
We know from Friedlander's result (see \cite{Fri05} or \cite[Theorem 3.1.7]{BK13}) that, given a compact graph 
such that $N(v)>1$ for at least one vertex $v$ and 
with Hamiltonian $H_0 = - \partial_{x}^2$, the set
\[\{(L_1, \dots, L_N) \in [0,+\infty)^N \mid \text{all eigenvalues of } - \partial_{x}^2\ \text{have multiplicity one}\},\]
is residual in the topological space $[0,+\infty)^N$. 
The conclusion then follows from \Cref{lem:simplegraph}.
\end{proof}

\begin{remark}
    By using \cite{BL17}, one can generalize \Cref{cor:genericpointphaseequivalentgraph} to quantum graphs with so-called $\delta$-type condition, including in particular mixed Dirichlet and Neumann--Kirchoff conditions at the vertices.
\end{remark}

\begin{cor}
\label{cor:genericpointphaseequivalentgraphdirichletatone}
    Let $\mathcal G$ be a compact graph with Hamiltonian $H_0 = - \partial_{x}^2$ and Dirichlet boundary conditions    
    at at least one Dirichlet vertex of $\mathcal G$.  Then, there exists a subsequence of eigenfunctions among  
     which no pair share the same modulus.
\end{cor}

\begin{proof}
    We use \cite[Theorem 2]{PT21} to obtain that there exists an increasing sequence of eigenvalues $(\lambda_{k_l})_{l \geq 1}$ of multiplicity one.
    Therefore, by \Cref{lem:simplegraph}, 
    the corresponding eigenfunctions cannot share the same modulus.
\end{proof}

\begin{remark}
    By using \cite[Theorem 4]{PT21}, one can generalize \Cref{cor:genericpointphaseequivalentgraphdirichletatone} to Hamiltonians of the type $H_0 = - \partial_{x}^2 + V$, where $V$ is edgewise constant.
\end{remark}

\subsection{Bilinear control of the Schr\"odinger equation in quantum graphs}
\label{sec:bilinearcontrolquantum}

Let $Q=(Q_1, \dots, Q_m)\in L^{\infty}(\mathcal G;\R)^m$ be the $m$-tuple potentials of interaction and let us consider the bilinear controlled Schr\"odinger equation
\begin{equation}
\label{eq:ScBilinearGraph}
\left\{
\begin{array}{ll}
	 i \partial_t \psi = - \partial_{x}^2 \psi + V(x) \psi + \langle u(t), Q(x) \rangle_{\R^d} \psi&   (t,x)\in (0,T)\times \mathcal G,\\
 \psi(t,v)=0&  (t,v) \in (0,T)\times \mathcal{V}_D,\\ 
 \sum_{j\in N(v)}\partial_x \psi^j(t,v)=0&  (t,v) \in (0,T)\times (\mathcal{V}_N \cup \mathcal{V}_i),\\
 \psi(0,\cdot)=\psi_0 & \text{in}\ \mathcal G.
	\end{array}
\right.
\end{equation}
As for Riemannian manifolds, see \Cref{sec:bilinearquantumsystemsmanifolds}, the well-posedness of \eqref{eq:ScBilinear} is a consequence of \cite{BMS82}. The notions of mild solutions, reachable sets, controllability are then easily adapted. In particular, we keep the same notations. The interior of $\mathcal{R}(\psi_0)$ for the topology of $\mathrm{Dom}(H_0) \cap \mathcal S$ is empty. For positive results of exact or approximate controllability in subspaces of $L^2(\mathcal G)$, we refer to \cite{Duc20} and \cite{Duc21}.

Assume that $Q_1, \dots, Q_m \in C^0(\mathcal G;\R)$ and the restriction of $Q_1, \dots, Q_m$ to each edge of $\cal G$ is  $C^{\infty}$ up to the boundary. Let us define as before the sequence of positive cones
\begin{equation}
\label{eq:defH0subspacegraph}
     \mathcal{H}_0 =  \left\{\varphi \in \text{span}\{Q_1, \dots, Q_m\} \mid \varphi \mathrm{Dom}(H_0)  \subset \mathrm{Dom}(H_0) \right\} \subset L^2(\mathcal G),
\end{equation}
and by recurrence 
\begin{multline}
\label{eq:defHNgraph}
\mathcal H_{N+1} = \left\{ \varphi \in \mathcal{H}_N +\left\{- \alpha g(\nabla_g \psi,\nabla_g \psi) \mid \pm\psi \in \mathcal{H}_N,\ \alpha \geq 0\right\}\mid \varphi \mathrm{Dom}(H_0)  \subset \mathrm{Dom}(H_0) \right\}.
\end{multline}
Finally, we define 
\begin{equation}
\label{eq:defHinftygraph}
    \mathcal{H}_{\infty} = \bigcup_{N \geq 0} \mathcal{H}_N.
\end{equation}

We have the following result, that is an adaptation of \cite[Theorem A]{DN21} in the context of quantum graphs. 
\begin{theorem}
\label{th:smalltimeapproximatequantumgraphs}
    Assume that $\mathcal{H}_{\infty}$ is dense in $ L^2(\mathcal G; \R)$. Equation \eqref{eq:ScBilinearGraph} is small-time isomodulus 
    approximately controllable, i.e., 
    \begin{equation*}
        \{  e^{i \phi} \psi_0\mid \phi \in L^2(\mathcal G;\R)\} \subset \overline{\mathcal{R}_0({\psi_0})},\qquad \mbox{for every $\psi_0 \in \mathcal{S}$}.
    \end{equation*}
\end{theorem}
The proof of \Cref{th:smalltimeapproximatequantumgraphs} is given in \Cref{sec:applicationquantumgraphs}.

In the next two parts, we consider two examples of quantum graphs in which \Cref{th:smalltimeapproximatequantumgraphs} can be applied. 
In the first example we exhibit a quantum graph with topology different from that of the circle and admitting eigenfunctions that share the same modulus. 
Each of these eigenfunctions corresponds to an eigenvalue of multiplicity $3$. Note also that in this case the Schrödinger operator admits eigenfunctions corresponding to simple eigenvalues that do not share the same modulus between themselves because of \Cref{lem:simplegraph} and actually do not share the same modulus with all other eigenfunctions. This highlights the variety of situations that may occur. In the second example we exhibit a quantum graph whose corresponding eigenvalues are all nonsimple (apart from the ground state) and that does not admit eigenfunctions sharing the same modulus and corresponding to different eigenvalues.

For convenience, we use in the following two sections the following notation: for every $\lambda\in \R$, $c_\lambda:[0,2\pi]\to \R$, $s_\lambda:[0,2\pi]\to \C$, $\eta_\lambda:[0,2\pi]\to \R$ denote the functions defined by 
\[c_\lambda(x)=\cos(\lambda x),\qquad s_\lambda(x)=\sin(\lambda x),\qquad \eta_\lambda(x)=e^{i\lambda x}.\] 
According to this notation,  $\eta_\lambda=c_\lambda+i s_\lambda$. 

\subsection{The example of the eight graph}
\label{subsection-graph-8}

We consider here the compact graph $\mathcal G$ with $N=2$ edges $e_1,e_2$, both  of length $2 \pi$, with $M=1$ vertex. 
Namely, $\mathcal G$ is composed of two loops having the same length $L=2 \pi$ that are linked to the same vertex, see \Cref{fig:eightgraph}.
\begin{figure}
 \centering
\includegraphics[width=8cm]{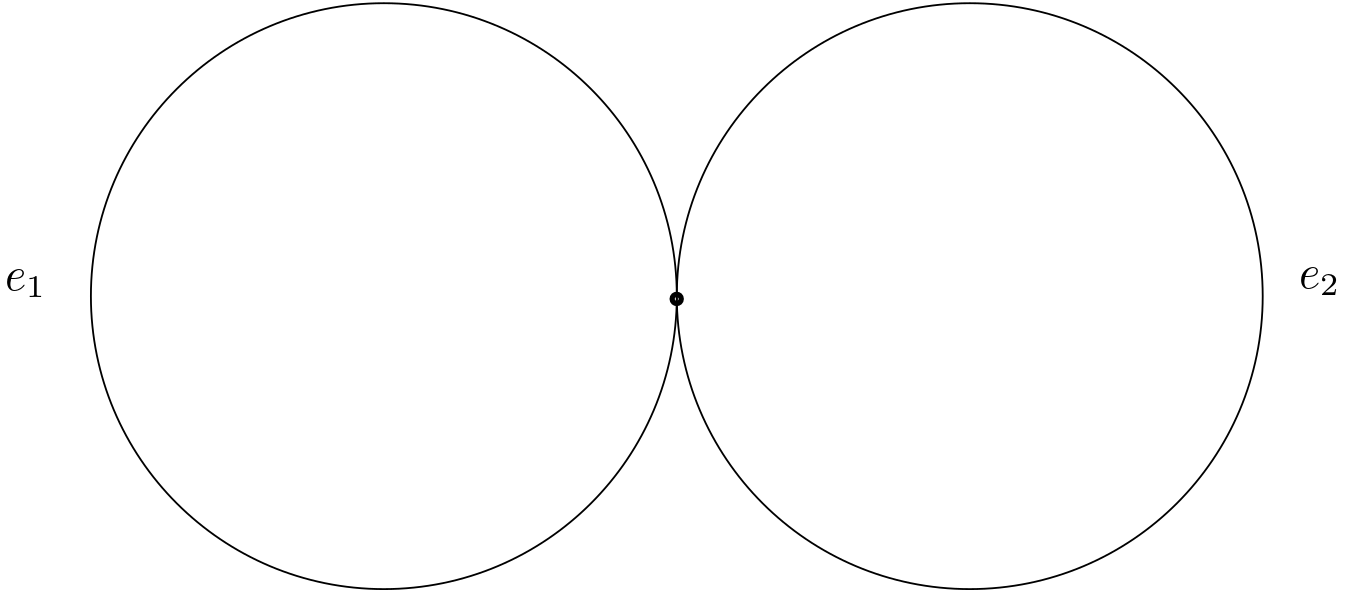}
    \caption{The eight graph }
    \label{fig:eightgraph}
\end{figure}

We consider the Hamiltonian $H_0 = -\partial_{x}^2$ on $\mathcal G$ and Neumann--Kirchoff boundary conditions at the single vertex.

We have the following result.
\begin{prop}
    The eigenvalues and eigenfunctions are given by
 \[       \lambda_0 = 0 \mbox{ of multiplicity 1},\ \phi_0 = (1,1),\]
 plus, for every $k \geq 0$,
 \[        \lambda_{k,o} = \left(\frac{2k+1}{2}\right)^2  \mbox{ of multiplicity 1}, \  \phi_{k,o} = \left( 
 s_{\frac{2k+1}{2}},-s_{\frac{2k+1}{2}}
\right),\]
 and finally, for $k\ge 1$, 
\begin{align*}   
\lambda_{k,e} &= k^2 \mbox{ of multiplicity 3},\nonumber\\
         \phi_{k,e,1} &= 
         (c_k,c_k),\ \phi_{k,e,2} =
         (s_k,0), 
      \ \phi_{k,e,3}= (0,s_k). 
    \end{align*}
    
\end{prop}
\begin{proof}
We know that the eigenvalues are of the form $\lambda^2$, $\lambda \in \R$, because the Hamiltonian operator is nonnegative. As a consequence, on each edge, the eigenfunction is a solution to
\begin{equation*}
    \partial_{x}^2 + \lambda^2 \psi = 0.
\end{equation*}
So, the eigenfunctions are of the form
\begin{equation*}
    \phi = (\phi^1, \phi^2) =  
    ( A_1 c_\lambda 
    + B_1 s_\lambda, 
    A_2 c_\lambda 
    + B_2 s_\lambda). 
\end{equation*}
We then write the boundary conditions, that are the continuity and the Neumann--Kirchoff  conditions at the vertex, leading to
\begin{align}
   & \phi^1(0) = \phi^2(0) \Rightarrow A= A_1=A_2,\\
   & \phi^1(2 \pi) = \phi^2(2 \pi) \Rightarrow B_1 \sin (2 \pi \lambda) = B_2 \sin (2 \pi \lambda),\label{B1B2}\\
   &  \phi^1(0) = \phi^1(2 \pi) \Rightarrow A = A \cos(2 \pi \lambda )+ B_1 \sin( 2 \pi \lambda),\label{eq:continuityeight}\\
   & \partial_{x} \phi^{1}(0) - \partial_{x}\phi^{1}(2 \pi) + \partial_{x}\phi^{2}(0) - \partial_{x}\phi^{2}(2 \pi) = 0\nonumber \\ & \ \Rightarrow B_1 \lambda - (-A \lambda \sin ( 2 \pi \lambda) + B_1 \lambda \cos(2 \pi \lambda)) + B_2 \lambda - (-A \lambda \sin ( 2 \pi \lambda) + B_2 \lambda \cos(2 \pi \lambda)) = 0.\label{eq:kirchoffeight}
\end{align}

If $\lambda \notin \N^{*}/2$, then we obtain from \eqref{B1B2} that $B=B_1 = B_2$. We therefore deduce from \eqref{eq:kirchoffeight} that
\begin{equation*}
   2 B  (\lambda - \lambda \cos(2 \pi \lambda)) + 2 A \lambda \sin(2 \pi \lambda) = 0.
\end{equation*}
Taking into account also \eqref{eq:continuityeight}, we find that $(A,B)$ satisfies the linear system
\begin{equation*}
  \begin{pmatrix}
1- \cos(2 \pi \lambda)  & - \sin( 2 \pi \lambda)\\
\sin( 2 \pi \lambda) & 1- \cos(2 \pi \lambda)
\end{pmatrix}    \begin{pmatrix}
A\\
B
\end{pmatrix}  
= \begin{pmatrix}
0\\
0
\end{pmatrix}  .
\end{equation*}
The determinant of the matrix is different from $0$,  leading to $A=B=0$.

So necessarily $\lambda \in \N^{*}/2$. If $\lambda = k/2$ with $k$ odd, we  find that $A= 0$ by \eqref{eq:continuityeight} and 
$B_1 = - B_2$ by the Neumann--Kirchoff condition \eqref{eq:kirchoffeight} . Finally, if $\lambda = k/2$ with $k$ even, we find that all the conditions are satisfied if $A_1=A_2$, that is, $A,B_1,B_2$ are free parameters. In the case $k=\lambda=0$ the space parameterized by $A,B_1,B_2$ is one-dimensional, otherwise it has dimension 3.

To conclude,
we have 
that $\lambda = k/2$, $k$ odd, is a simple eigenvalue with corresponding eigenfunction $\phi_{k,o}$ 
and $\lambda = k/2$, $k$ even, corresponds
either to a simple eigenvalue for $k=0$ or
to an 
eigenvalue of multiplicity $3$ for $k>0$, associated with eigenfunctions of the type $\phi_{k,e,j}$, $j=1,2,3$. 
\end{proof}

\begin{prop}
\label{prop:pointphaseeightgraph}
    For every $k,l \geq 0$, $k \neq l$,  there exists an eigenfunction corresponding to $k^2$ 
    and an eigenfunction corresponding to $l^2$
    that share the same modulus.

    For every $k \geq 1$, there exist two $\C$-linearly independent eigenfunctions corresponding to $k^2$ that 
    share the same modulus.

    For every $k,l \geq 0$, $k \neq l$, two eigenfunctions corresponding to $\left(\frac{2k+1}{2}\right)^2$ and $\left(\frac{2l+1}{2}\right)^2$ 
do not share the same modulus.

    For every $k \geq 0$, $l \geq 1$, two eigenfunctions corresponding to $\left(\frac{2k+1}{2}\right)^2$ and $l^2$ 
    do not share the same modulus.
\end{prop}
\begin{proof}
For the first point, let us remark that, for $k\ge 1$,
\begin{equation*}
    \phi_{k,e,1} + i \phi_{k,e,2} + i \phi_{k,e,3} = (\eta_k, \eta_k),
\end{equation*}
so that $|\phi_{k,e,1} + i \phi_{k,e,2} + i \phi_{k,e,3}|=(1,1)$. 
 
For the second point, let us notice that
\begin{equation*}
    - \phi_{k,e,1} + i \phi_{k,e,2} + i \phi_{k,e,3} = (\eta_{-k},\eta_{-k}).
\end{equation*}

For the third point, we use the fact that the eigenvalues are simple and \Cref{lem:simplegraph}.

For the fourth point, let us argue by contradiction. Assume that there exist $k \geq 0$, $l \geq 1$, $C_1, C_2 \in \C$ such that
\begin{equation*}
    \left|\sin\left(\frac{2k+1}{2}x\right)\right| = \left|C_1 \cos(lx) + C_2 \sin(lx)\right|,\qquad x\in [0,2\pi].
\end{equation*}
Taking $x=0$ leads to $C_1=0$ and then taking $x=\pi$ one gets a contradiction.
\end{proof}

We define the potentials
\begin{align*}
    Q_1 &= (1, 1),\ Q_2 = (c_1,c_1),
    \ Q_3 = (s_1,0),\ 
    Q_4 =(0,s_1),\ 
    Q_5  =\left(s_{\frac12},-s_{\frac12}\right),\  
    Q_6 =\left(c_{\frac12},c_{\frac12}\right).
\end{align*}

We have the following lemma.
\begin{lemma}
\label{lem:testdensityHinfty}
    The set $\mathcal{H}_{\infty}$ is dense in $L^2(\mathcal G;\R)$. 
\end{lemma}
\begin{proof}

The main ingredient of the proof consists in establishing that
    \begin{equation*}
        \{\phi_0\} \cup \{\phi_{k,o}\mid k \geq 0\} \cup \{\phi_{k,e,j}\mid k \geq 1,\; j=1,2,3\} \subset \mathcal{H}_{\infty}.
    \end{equation*}
    
    First note that these functions stabilize $\text{Dom}(H_0)$ and that  $\phi_0=Q_1$, $\phi_{0,o}=Q_5$, $\phi_{1,e,1}=Q_2$, $\phi_{1,e,2}=Q_3$, $\phi_{1,e,3}=Q_4$ belong to $\mathcal{H}_0 \subset \mathcal{H}_{\infty}$. 

Let us then prove that $\phi_{k,e,j} \in  \mathcal{H}_{\infty}$ 
for every $k \geq 1$ and $j=1,2,3$.
Notice that the property is true for $k=1$. 
The argument is based on the trigonometric formulas
\begin{equation}
\label{eq:simpletrigo}
\cos(mx)^2 + \sin(mx)^2= 1,\ \cos(lx)^2 +  \sin(lx)^2= 1,
\end{equation} 
\begin{multline}
\label{first:mll}
    \cos((m+l)x) = \cos(mx) \cos(lx) - \sin(mx) \sin(lx)\\
    = \left(- \frac{1}{2} (\cos(mx) - \cos(lx))^2 + \frac{1}{2} \cos(mx)^2  + \frac{1}{2} \cos(lx)^2\right) \\
    +\left( -\frac{1}{2}(\sin(mx) + \sin(lx))^2 + \frac{1}{2} \sin(mx)^2 +  \frac{1}{2} \sin(lx)^2\right),
\end{multline}
and
\begin{multline}\label{second:mll}
    \sin((m+l)x) = \sin(mx) \cos(lx) + \sin(lx) \cos(mx)\\
    = \left(- \frac{1}{2} (\sin(mx) - \cos(lx))^2 + \frac{1}{2} \sin(mx)^2 +  \frac{1}{2} \cos(lx)^2 \right) \\
    + \left(-\frac{1}{2} (\sin(lx) - \cos(mx))^2 + \frac{1}{2} \sin(lx)^2 +  \frac{1}{2} \cos(mx)^2\right),
\end{multline}
holding for every $l,m\in \R$. 
For $m,l\ne 0$ we can rewrite \eqref{eq:simpletrigo} as  
\begin{equation}
c_m^2 = 1 - \left(\frac{c_m'}{m}\right)^2,\ s_m^2 = 1 - \left(\frac{s_m'}{m}\right)^2, \qquad c_l^2 = 1 - \left(\frac{c_l'}{l}\right)^2,\ s_l^2 = 1 - \left(\frac{s_l'}{l}\right)^2,
\end{equation}
and \eqref{first:mll} as
\begin{equation}
\label{eq:findcos}
    2c_{m+l}=   - \left(\frac{s_m'}{m}  - \frac{s_l'}{l}\right)^2 + c_m^2 +   c_l^2 -   \left(\frac{c_m'}{m}  + \frac{c_l'}{l} \right)^2 +  s_m^2 +   s_l ^2.
\end{equation}
Assume for now that 
$\pm(c_m,c_m)$, $\pm(c_l,c_l)$, 
$\pm(s_m,s_m)$, $\pm(s_l,s_l)$ 
belong to $\mathcal{H}_{\infty}$.
Then the same is true first for $ (c_m^2,c_m^2)$, $(c_l^2,c_l^2)$, 
$(s_m^2,s_m^2)$, $(s_l^2,s_l^2)$ 
and then  for 
$(c_{m+l},c_{m+l})$. 
Similarly, noticing that, in analogy to \eqref{first:mll},  
\[-c_{m+l}=\frac12(-(s_m-s_l)^2+s_m^2+s_l^2)+\frac12(-(c_m+s_l)^2+c_m^2+c_l^2),\]
one deduces that also 
$-(c_{m+l},c_{m+l})$ is in 
$\mathcal{H}_{\infty}$. 
A similar reasoning based on \eqref{second:mll}
also shows that 
$\pm(s_{m+l},s_{m+l}) \in \mathcal{H}_{\infty}$.

Taking $l=1$ and by recurrence on $m\ge 1$ (integer) this proves that $\phi_{k,e,1}$ is in  $\mathcal{H}_{\infty}$ 
for every $k \geq 1$.

We proceed in the same way to prove that $\pm(s_{m+l},-s_{m+l})
\in \mathcal{H}_{\infty}$ provided that 
$\pm(c_m,c_m)$, $\pm(c_l,c_l)$, $\pm(s_m,0)$, $\pm(s_l,0)$, $\pm(0,s_m)$, $\pm(0,s_l)$
belong to $\mathcal{H}_{\infty}$. 
Hence $\phi_{k,e,j}$, $j=2,3$, is in  $\mathcal{H}_{\infty}$ 
for every $k \geq 1$.

The proof for the functions $\phi_{k,o}$ 
uses the fact that
$(s_1, -s_1)$, $(c_1, c_1)$, $(s_{1/2}, -s_{1/2})$, $(c_{1/2}, c_{1/2})$ belong to $\mathcal H_{\infty}$ and that
\begin{equation*}
\label{eq:findsin}
   \pm 2 s_{m+\frac12} 
    =  - \left(\frac{c_m'}{m} \pm  2s_{\frac12}'\right)^2 +   s_{\frac12}^2+  c_m^2  -   \left(2 c_{\frac12}' \pm \frac{s_m'}{m}\right)^2 +    s_m^2+   c_{\frac12}^2 ,
\end{equation*}
which follows 
taking $m\ne 0$ and $l=1/2$ in  \eqref{second:mll} and in its equivalent reformulation
\[ -s_{m+l}= 
\left(-\frac{1}{2} (s_m +c_l)^2 + \frac{1}{2} s_m^2 +  \frac{1}{2} c_l^2 \right)     + \left(-\frac{1}{2} (s_l +c_m)^2 + \frac{1}{2} s_l^2 +  \frac{1}{2} c_m^2\right).\]
The conclusion is then obtained by recurrence on $m\ge 1$ (integer).
\end{proof}

\begin{theorem}
\label{th:resultisographeight}
   Equation \eqref{eq:ScBilinearGraph} is small-time isomodulus
   approximately controllable.
   Moreover, for every $k,l \in \Z$, we have that
   \begin{equation*}
       ( \eta_{l}, \eta_{l}) \in \overline{\mathcal{R}_0( \eta_{k}, \eta_{k})}.
   \end{equation*}
\end{theorem}
\begin{proof}
    This is a direct consequence of \Cref{th:smalltimeapproximatequantumgraphs}, \Cref{prop:pointphaseeightgraph}, and \Cref{lem:testdensityHinfty}.
\end{proof}

\subsection{The example of the graph with three branches}
\label{subsection-graph-3branches}

We consider now as $\mathcal G$ the compact graph with $N=3$ edges $e_1,e_2, e_3$, each of them of length $2 \pi$, with $M=2$ vertices $v_1,v_2$ and no loops. 
Namely, $\mathcal G$ is composed of three edges having the same length $L=2 \pi$ that are linked to the same vertices. We orient the three edges 
 from $v_1$ to $v_2$.
See Figure~\ref{fig:Perforation}.

\begin{figure}
\centering
\includegraphics[width=7cm]{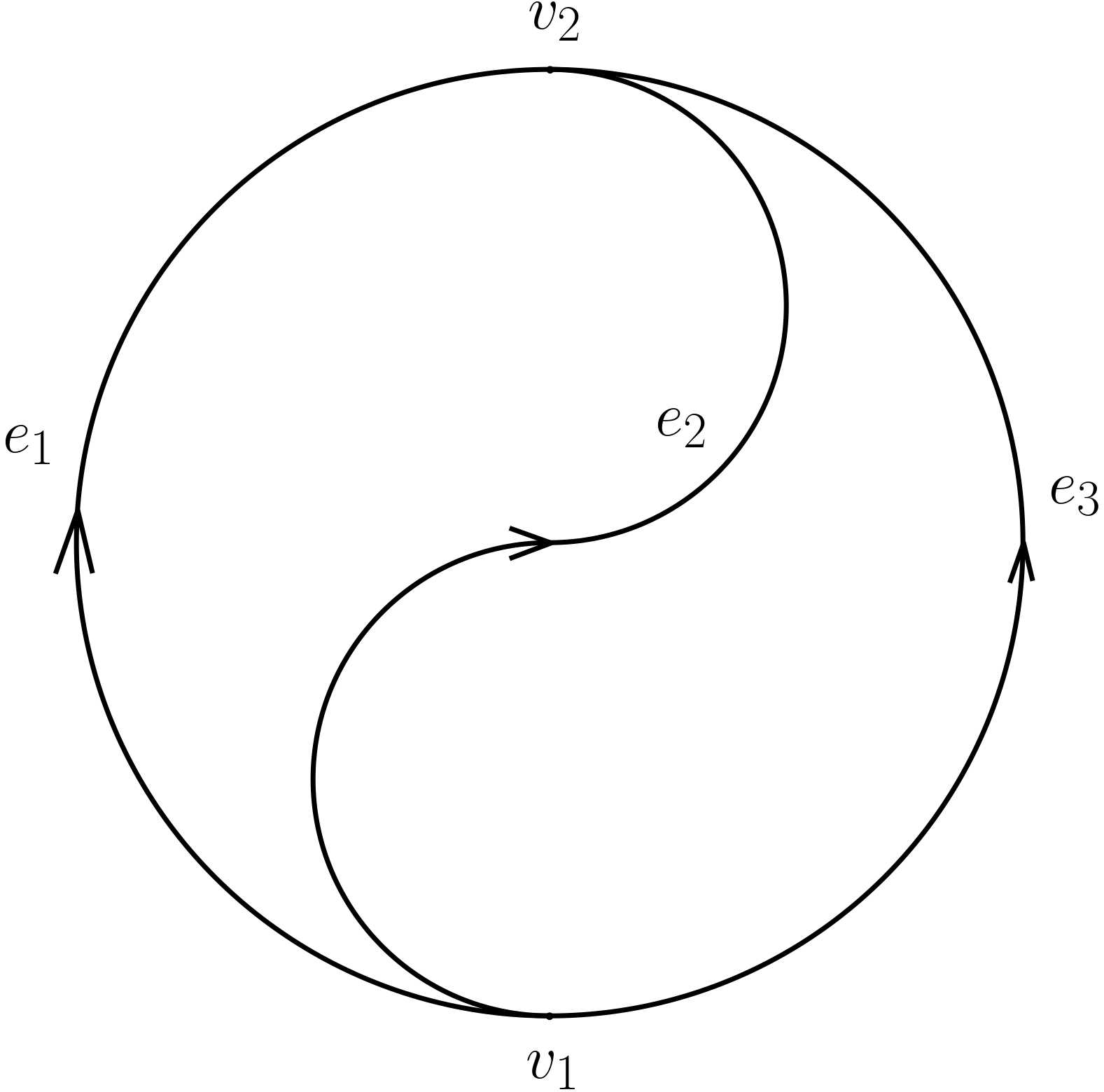}
\caption{The graph with three branches. }
\label{fig:Perforation}
\end{figure}

We consider the Hamiltonian $H_0 = -\partial_{x}^2$ on $\mathcal G$ and Neumann--Kirchoff boundary conditions.

We have the following result.
\begin{prop}
    The eigenvalues and eigenfunctions are given by
\[\lambda_0 = 0,\ \phi_0 = 1,\]
and, for $k \geq 1$,      
    \begin{align*}
  \lambda_{k} &= \left(\frac{k}{2}\right)^2\quad   \text{of multiplicity three},\\
        \phi_{k,1}&= \left( c_{\frac{k}{2}}, c_{\frac{k}{2}},  c_{\frac{k}{2}} \right),\quad 
       \phi_{k,2} = \left( -s_{\frac{k}{2}}, s_{\frac{k}{2}}, 0 \right),\quad 
        \phi_{k,3} = \left( -s_{\frac{k}{2}}, 0, s_{\frac{k}{2}} \right).
    \end{align*}
\end{prop}

\begin{proof}
    We know that the eigenvalues are of the form $\lambda^2$, $\lambda \in \R$, because the Hamiltonian operator is nonnegative. As a consequence, on each edge, the eigenfunction is a solution to
\begin{equation*}
    \partial_{x}^2 + \lambda^2 \psi = 0.
\end{equation*}
So, the eigenfunctions are of the form
\begin{equation*}
    \phi = (\phi^j)_{1 \leq j \leq 3} =  ( A_j c_\lambda     + B_j s_\lambda    )_{1 \leq j \leq 3}.
\end{equation*}
The case $\lambda=0$ being easy, let us focus on the case $\lambda \neq 0$.

We then write the boundary conditions, that are the continuity and the Neumann--Kirchoff 
 conditions at each vertex, leading to
\begin{align}
   & \phi^1(0) = \phi^2(0) = \phi^3(0) \Rightarrow A_1 = A_2 = A_3 = A,\\
   & \phi^1(2 \pi) = \phi^2(2 \pi) = \phi^3(2\pi) \Rightarrow B_1 \sin(2 \pi\lambda) = B_2 \sin(2 \pi\lambda) = B_3 \sin(2 \pi\lambda),\label{eq:B1B2B3}\\
   &  \partial_{x}\phi^1(0)  + \partial_{x}\phi^2(0) + \partial_{x}\phi^3(0)  = 0
    \Rightarrow 
    B_1+B_2+B_3=0,
    \label{eq:kds}\\
   & \partial_{x}\phi^1(2 \pi)  + \partial_{x}\phi^2(2 \pi) + \partial_{x}\phi^3(2 \pi)  = 0\nonumber 
   \\
   & \Rightarrow  -3 A \lambda \sin(2 \pi\lambda ) + ( B_1 + B_2 + B_3) \lambda \cos(2 \pi\lambda)  = 0
    \Rightarrow A \sin(2 \pi\lambda) = 0.\label{eq:kirchoffthree2}
\end{align}
If $\sin(2 \pi\lambda) \neq 0$, i.e., $\lambda \notin \N^{*}/2$, then we obtain from \eqref{eq:B1B2B3} and \eqref{eq:kirchoffthree2} that $B_1 = B_2 = B_3 = B$ and $A=0$. Then $3B=0$ by \eqref{eq:kds}, so that also $B=0$. This case can then be excluded.

If $\sin(2 \pi\lambda) = 0$ then, all the conditions are fulfilled provided that $A=A_1=A_2=A_3$ and $B_1=-B_2-B_3$. This corresponds to the eigenfunctions introduced in the statement.
\end{proof}

We remark now that the eigenvalues are not simple so there is no  a priori obstruction to the isomodulus approximate controllability between two eigenfunctions corresponding to different energy levels. Another obstruction actually holds, as presented below.

\begin{prop}
    For every $k,l \geq 0$, $k \neq l$, two eigenfunctions corresponding to $\left(\frac{k}{2}\right)^2$ and $\left(\frac{l}{2}\right)^2$ 
    do not share the same modulus.
\end{prop}
\begin{proof}
  Let us proceed by contradiction,  supposing that two distinct energy levels are connectable. Since $k$ and $l$ are different, one of the two, say $k$, is larger than zero. 
  A general eigenfunction  corresponding to the eigenvalue $(k/2)^2$ can be written as 
  $$\phi_k = A_1 \phi_{k,1} + A_{2} \phi_{k,2} + A_{3} \phi_{k,3},$$
  with $A_1,A_2,A_3\in\C$ and $|A_1|^2+|A_2|^2+|A_3|^2\ne0$. 
We know from Lemma~\ref{lem:inab} that the modulus of $\phi_k$ on each edge 
is a constant so we have
\begin{align}
    \left|A_1 \cos\left(\frac{k}{2} x\right) - A_2 \sin\left(\frac{k}{2} x\right) - A_3 \sin\left(\frac{k}{2} x\right)\right|^2 &= C_1,\qquad x\in (0,2\pi),\label{eq:A1}\\
   \left|A_1 \cos\left(\frac{k}{2} x\right) + A_2 \sin\left(\frac{k}{2} x\right) \right|^2 &= C_2,\qquad x\in (0,2\pi),\label{eq:A2}\\
   \left|A_1 \cos\left(\frac{k}{2} x\right) + A_3\sin\left(\frac{k}{2} x\right) \right|^2 &= C_3,\qquad x\in (0,2\pi),\label{eq:A3}
\end{align}
for some $C_1,C_2,C_3\ge 0$. 

If $A_1=0$ then we deduce from \eqref{eq:A2} and \eqref{eq:A3} that also $A_2,A_3=0$. We can then assume without loss of generality that $A_1 = 1$. 
Taking $x=0$, we get that $C_1 = C_2 = C_3 =1$ so that we can rewrite \eqref{eq:A1}--\eqref{eq:A3} as 
 \begin{align*}
   \left| \cos\left(\frac{k}{2} x\right) - A_2 \sin\left(\frac{k}{2} x\right) - A_3 \sin\left(\frac{k}{2} x\right)\right|^2 &= 1,\\
 \left| \cos\left(\frac{k}{2} x\right) + A_2 \sin\left(\frac{k}{2} x\right) \right|^2 &= 1,\\
  \left|  \cos\left(\frac{k}{2} x\right) + A_3 \sin\left(\frac{k}{2} x\right) \right|^2 &= 1.
\end{align*} 

 Taking $x=\frac{\pi}{4k}$ (so that $\frac{k}{2}x=\frac{\pi}2$), we get that 
 $|A_2 + A_3| = |A_2| = |A_3|=1$. 
Take now 
$x=\frac{\pi}{8k}$ (so that $\frac{k}{2}x=\frac{\pi}4$) and deduce that $|1-A_1-A_2|=|1+A_2|=|1+A_3|=1$. 
Since the only $z\in \C$ such that $|z|=1$ and $|1+ z|=2$ is $z=1$, we deduce that 
$-A_2 - A_3 = A_2 = A_3=1$, leading to a contradiction.
\end{proof}

\appendix

\section{Small-time isomodulus approximate controllability}
\label{sec:appendixsmalltimepointphase}

\subsection{Abstract setting}

Let $\mathcal H$ be an infinite-dimensional Hilbert space, $H_0$ be an unbounded self-adjoint operator with domain $\mathrm{Dom}(H_0) \subset \mathcal H$, $H_1, \dots, H_m$ be bounded self-adjoint operators on $\mathcal{H}$. Let us consider the  abstract bilinear control system
\begin{equation}
\label{eq:ScBilinearAbstract}
\left\{
\begin{array}{ll}
i \partial_t \psi(t) = H_0 \psi(t) +  \sum_{j=1}^m u_j(t) H_j \psi(t),& t \in  (0,+\infty),\\
\psi(0) = \psi_0.
\end{array}
\right.
\end{equation}
We have the following well-posedness result according to \cite{BMS82}. For every $T>0$, $\psi_0 \in \mathcal{H}$, and $u=(u_1,\dots, u_m) \in L^2(0,T;\R^d)$, there exists a unique mild solution $\psi \in C([0,T];\mathcal{H})$ of \eqref{eq:ScBilinearAbstract}, i.e.,
\begin{equation*}
    \psi(t) = e^{-it H_0} \psi_0 + \int_0^t e^{- i (t-s) H_0} \left(\sum_{j=1}^m u_j(t) H_j\right) \psi(s) ds,\qquad \forall 
    t \in [0,T].
\end{equation*}
If we assume furthermore that $\psi_0 \in \mathcal{S} = \{\psi \in \mathcal{H}\mid \|\psi\|_{\mathcal{H}} = 1\}$, then  $\psi(t) \in \mathcal{S}$ for every $t \in [0,T]$. In the sequel, for $\psi_0 \in \mathcal{H}$ and $u \in L^2(0,T;\R^d)$, the associated solution $\psi \in C([0,T];\mathcal{H})$ of \eqref{eq:ScBilinearAbstract} will be denoted by $\psi(\cdot; \psi_0, u)$.

The following result is exactly \cite[Theorem 8]{CP22}.
\begin{theorem}\label{thm:abstract}
Let $S$ be a bounded self-adjoint operator satisfying 
\begin{align}
&[S,H_j]=0, \quad j=1,\dots,m,\label{eq:commutationcontrol}\\
&S\mathrm{Dom}(H_0)\subset \mathrm{Dom}(H_0),\quad {\rm ad}^3_S(H_0)\mathrm{Dom}(H_0)=0.\label{eq:liecrochet}
\end{align}
Then, for each $\psi_0\in\mathcal{H}$ and $u = (u_1, \dots, u_m) \in \R^m$, the following limit holds in $\mathcal{H}$
\begin{align}\label{eq:limitabstract}
\lim_{\delta \to 0}&e^{-i\delta^{-1/2}S}\exp\left(-i\delta\left(H_0+\sum_{j=1}^m\frac{u_j}{\delta}H_j\right)\right)e^{i\delta^{-1/2}S}\psi_0\nonumber \\
=&\exp\left(\frac{i}{2} {\rm ad}^2_S(H_0)-i\sum_{j=1}^m u_jH_j\right)\psi_0,
\end{align}
where ${\rm ad}^0_AB=B$, ${\rm ad}_AB=[A,B]=AB-BA$, and ${\rm ad}^{n+1}_AB=[A,{\rm ad}^{n}_AB]$. 
\end{theorem}

\subsection{Applications to manifolds with boundary}

We start from the setting of \Cref{sec:bilinearquantumsystemsmanifolds}.
Consider the following asymptotic result.

\begin{theorem}
\label{thm:limitDirichletNeumann}
For every $\psi_0\in L^2(M)$, $(u_1,\dots,u_m)\in\mathbb{R}^m$,  and $\varphi\in C^\infty(M;\mathbb{R})$ such that $\varphi \text{Dom}(H_0) \subset\text{Dom}(H_0)$, the following limit holds in $L^2(M)$
\begin{align}
\notag
\lim_{\delta \to 0}e^{-i\delta^{-1/2}\varphi}\exp&\left(-i\delta\left(-\Delta_g+V+\sum_{j=1}^m\frac{u_j}{\delta}Q_j\right)\right)e^{i\delta^{-1/2}\varphi}\psi_0\\
=&\exp\left(-i g(\nabla_g \varphi,\nabla_g \varphi)-i \sum_{j=1}^mu_jQ_j\right)\psi_0.\label{eq:limitdirichlet}
\end{align}
\end{theorem}
This is an adapted version of \cite[Theorem 2]{CP22} to the case of a manifold with boundary. The result is a direct application of \Cref{thm:abstract}. Indeed, let us take $\mathcal H = L^2(M)$, $H_0 = - \Delta_g + V$, $H_j = Q_j$ for every $j\in\{1,\dots,m\}$, and let $S$ be the bounded multiplication operator by $\varphi$ on $\mathcal H$. Then \eqref{eq:commutationcontrol} and \eqref{eq:liecrochet} are fulfilled. Moreover, we have
\begin{equation}
    \label{eq:computationcrochet}
    \frac{1}{2} {\rm ad}^2_S(H_0) = - g(\nabla_g \varphi,\nabla_g \varphi).
\end{equation}

For $Q_1, \dots, Q_m \in C^{\infty}(M;\R)$, we recall the notation \eqref{eq:defH0subspace}, \eqref{eq:defHNsubspacerecurrence}, and \eqref{eq:defHinfty}. We have the following result.  
\begin{theorem}
\label{thm:manifoldirichletneumann}
For every $\psi_0 \in L^2(M)$, we have
\begin{equation*}
    \{e^{i\phi}\psi_0\mid\phi\in\mathcal{H}_\infty\}\subset \overline{\mathcal R_{0}(\psi_0)}.
\end{equation*}
\end{theorem}
The proof of the theorem, included for completeness, is an adaptation of the ones of \cite[Theorem~3]{CP22} and \cite[Theorem 1]{DP23}, and works because, for every $N \geq 0$ and every  $\varphi \in \mathcal{H}_N$, $\varphi \text{Dom}(H_0) \subset\text{Dom}(H_0)$,  so the arguments of the original proof still make sense. Notice that this last property is included in the definitions \eqref{eq:defH0subspace}, \eqref{eq:defHNsubspacerecurrence} while this inclusion was automatically satisfied in the specific setting chosen for \cite[Theorem~3]{CP22}. Finally, note that \Cref{thm:manifoldirichletneumann} directly implies \Cref{th:smalltimepointphase}.

\begin{proof}[Proof of Theorem \ref{thm:manifoldirichletneumann}]
Let us introduce the following notation for the concatenation of two controls $u$ and $v$: given $\tau\ge 0$ such that $u$ is defined at least up to time $\tau$,  $u \diamond_{\tau}  v$ denotes the function
\begin{equation}
\label{eq:defconcatenation}
(u \diamond_{\tau} v) (t)= \begin{cases}
 u(t)&\quad \hbox{for} \ \ t\in[0,\tau),\\
 v(t-\tau)&\quad \hbox{for}\ \ t\geqslant\tau .
 \end{cases}
 \end{equation}
 
It is sufficient to prove by recurrence on $N\ge 0$ that for every 
$\psi_0\in L^2(M)$,
\begin{equation}
\label{eq:saturatedspacein}
\{e^{i\phi}\psi_0\mid \phi\in\mathcal{H}_N\}\subset  \overline{\mathcal R_{0}(\psi_0)}.
\end{equation}

\textit{Initialization:} The case $N=0$ is a straightforward consequence of \Cref{thm:limitDirichletNeumann} by taking $\varphi=0$. Indeed, we have that for every $(u_1,\dots,u_m)\in\mathbb{R}^m$
\begin{align}
\label{eq:limitdirichletapplication}
\lim_{\delta \to 0}\exp\left(-i\delta\left(-\Delta_g+V+\sum_{j=1}^m\frac{u_j}{\delta}Q_j\right)\right)\psi_0=\exp\left(-i \sum_{j=1}^mu_jQ_j\right)\psi_0\quad  \text{in}\ L^2(M).
\end{align}

\textit{Recurrence step:} Assume that for some $N \in \N$, we have \eqref{eq:saturatedspacein} for every $\psi_0\in L^2(M)$. Let us take  $\varepsilon>0$ and $\phi \in \mathcal{H}_{N+1}$. Then there exist $\phi_0, \dots, \phi_k \in \mathcal{H}_N$ and $\alpha_1, \dots, \alpha_k \in [0,+\infty)$ such that $-\phi_1,\dots,-\phi_k\in {\cal H}_N$ and
\begin{equation*}
 \phi = \phi_0 - \sum_{j=1}^k \alpha_j g (\nabla \phi_j, \nabla \phi_j).
\end{equation*}

Note that $\phi_1 \in \mathcal H_N$ stabilizes $\text{Dom}(H_0)$.
By applying \Cref{thm:limitDirichletNeumann} with $u=0$, we obtain that there exists $\gamma \in (0,T/3)$ such that 
\begin{equation*}
    \norme{e^{-i\gamma^{-1/2}\alpha_1^{1/2}\phi_1}e^{-i\gamma\left(-\Delta_g+V\right)}e^{i\gamma^{-1/2}\alpha_1^{1/2}\phi_1}e^{i \phi_0} \psi_0 - e^{ - i \alpha_1 g(\nabla \phi_1, \nabla \phi_1)} e^{i \phi_0} \psi_0}_{L^2(M)} < \varepsilon.
\end{equation*}

Then, by the inductive hypothesis, there exist $\tau_1 \in (0,T/3)$ and a control $u_1 \in L^2(0, \tau_1;\R^m)$ such that
\begin{equation}
\label{eq:firstapproximation}
    \norme{\psi(\tau_1; \psi_0, u_1) - e^{i\gamma^{-1/2}\alpha_1^{1/2}\phi_1}e^{i \phi_0} \psi_0}_{L^2(M)} < \varepsilon.
\end{equation}
By using that the semigroup $(e^{-it(-\Delta_g+V)})_{t \geq 0}$ is unitary and \eqref{eq:firstapproximation}, we obtain that for $\tau_2=\gamma$ and $u_2 = 0$ in $(0,\tau_2)$,
\begin{align}
   & \norme{\psi(\tau_2+\tau_1; \psi_0, u_1 \diamond_{\tau_1} u_2) - e^{-i\gamma\left(-\Delta_g+V\right)}e^{i\gamma^{-1/2}\alpha_1^{1/2}\phi_1}e^{i \phi_0} \psi_0}_{L^2(M)}\notag\\
   & =   \norme{e^{-i\gamma(-\Delta_g+V)} \psi(\tau_1; \psi_0, u_1) - e^{-i\gamma(-\Delta_g+V)} e^{i\gamma^{-1/2}\alpha_1^{1/2}\phi_1}e^{i \phi_0} \psi_0}_{L^2(M)}
   < \varepsilon. \label{eq:secondapproximation}
\end{align}
By using again the inductive hypothesis to the data $e^{-i\gamma\left(-\Delta_g+V\right)}e^{i\gamma^{-1/2}\alpha_1^{1/2}\phi_1}e^{i \phi_0} \psi_0$, we obtain that there exist a time $\tau_3 \in (0,T/3)$ and a control $u_3 \in L^2(0,\tau_3;\R^m)$ such that
\begin{multline}
\label{eq:thirdapproximation}
    \|\psi(\tau_3; e^{-i\gamma\left(-\Delta_g+V\right)}e^{i\gamma^{-1/2}\alpha_1^{1/2}\phi_1}e^{i \phi_0} \psi_0, u_3) \\
    - e^{-i\gamma^{-1/2}\alpha_1^{1/2}\phi_1}e^{-i\gamma\left(-\Delta_g+V\right)}e^{i\gamma^{-1/2}\alpha_1^{1/2}\phi_1}e^{i \phi_0} \psi_0 \|_{L^2(M)}
     < \varepsilon.
\end{multline}
We now set $\tau=\tau_1+\tau_2+\tau_3$ and $u = u_1 \diamond_{\tau_1} u_2 \diamond_{\tau_2} u_3 \in L^2(0,\tau;\R^m)$ and we have from \eqref{eq:secondapproximation} and \eqref{eq:thirdapproximation},
\begin{align*}
   & \norme{\psi(\tau; \psi_0,u) - e^{- i \alpha_1 g(\nabla \phi_1, \nabla \phi_1)} e^{i \phi_0} \psi_0}_{L^2(M)} \\
& \leq \norme{\psi(\tau_3;\psi(\tau_2+\tau_1; \psi_0, u_1 \diamond_{\tau_1} u_2), u_3) - \psi(\tau_3;e^{-i\gamma\left(-\Delta_g+V\right)}e^{i\gamma^{-1/2}\alpha_1^{1/2}\phi_1}e^{i \phi_0} \psi_0, u_3)  }_{L^2(M)}\\
& + \norme{ \psi(\tau_3;e^{-i\gamma\left(-\Delta_g+V\right)}e^{i\gamma^{-1/2}\alpha_1^{1/2}\phi_1}e^{i \phi_0} \psi_0, u_3) - e^{-i\gamma^{-1/2}\alpha_1^{1/2}\phi_1}e^{-i\gamma\left(-\Delta_g+V\right)}e^{i\gamma^{-1/2}\alpha_1^{1/2}\phi_1}e^{i \phi_0} \psi_0  }_{L^2(M)}\\
& + \norme{  e^{-i\gamma^{-1/2}\alpha_1^{1/2}\phi_1}e^{-i\gamma\left(-\Delta_g+V\right)}e^{i\gamma^{-1/2}\alpha_1^{1/2}\phi_1}e^{i \phi_0} \psi_0 -e^{ - i \alpha_1 g(\nabla \phi_1, \nabla \phi_1)} e^{i \phi_0} \psi_0  }_{L^2(M)}\\
& < 3 \varepsilon.
\end{align*}
By arbitrariness of $\varepsilon>0$, this proves that 
\begin{equation}
\label{eq:firstlimitHnafter}
    e^{- i \alpha_1 g(\nabla \phi_1, \nabla \phi_1)} e^{i \phi_0} \psi_0 \in \overline{\mathcal{R}_0(\psi_0)}.
\end{equation}
 We then iterate the argument by considering the limit \eqref{eq:limitdirichlet} with $\varphi = - \alpha_2^{1/2} \phi_2$ starting from the initial data $e^{i \alpha_1 g(\nabla \phi_1, \nabla \phi_1)} e^{i \phi_0} \psi_0$. Again, note that $\phi_2 \in \mathcal H_N$ stabilizes $\text{Dom}(H_0)$, therefore this leads to the possibility of  applying of \Cref{thm:limitDirichletNeumann}. 

Therefore, thanks to \eqref{eq:firstlimitHnafter}, one obtains 
\begin{equation*}
    e^{i(\phi_0 -
    \sum_{j=1}^k \alpha_j g (\nabla \phi_j, \nabla \phi_j))} \in \overline{\mathcal{R}_0(\psi_0)},
\end{equation*}
and the proof is complete. 
\end{proof}

\subsection{Application to quantum graphs}
\label{sec:applicationquantumgraphs}

We adopt here the setting of \Cref{sec:bilinearcontrolquantum}.
Consider the following asymptotic result.

\begin{theorem}
\label{thm:limitQuantum}
For every $\psi_0\in L^2(\mathcal G)$, $(u_1,\dots,u_m)\in\mathbb{R}^m$,  and $\varphi\in C^0(\mathcal G;\mathbb{R})$ such that the restriction of $\varphi$ to each edge is $C^{\infty}$ and $\varphi \text{Dom}(H_0) \subset\text{Dom}(H_0)$, the following limit holds in $L^2(\mathcal G)$
\begin{align}
\notag
\lim_{\delta \to 0}e^{-i\delta^{-1/2}\varphi}\exp&\left(-i\delta\left(-\partial_{x}^2+V+\sum_{j=1}^m\frac{u_j}{\delta}Q_j\right)\right)e^{i\delta^{-1/2}\varphi}\psi_0\\
=&\exp\left(-i (\partial_{x} \varphi)^2-i \sum_{j=1}^mu_jQ_j\right)\psi_0.\label{eq:limitgraph}
\end{align}
\end{theorem}
This is an adapted version of \cite[Theorem 2]{CP22} to the case of a quantum graph. The result is a direct application of \Cref{thm:abstract}. Indeed, let us take $\mathcal H = L^2(\mathcal G)$, $H_0 = - \partial_{x}^2 + V$, $H_j = Q_j$ for every $j\in\{1,\dots,m\}$, and let $S$ be the bounded multiplication operator by $\varphi$ on $\mathcal H$. Then \eqref{eq:commutationcontrol} and \eqref{eq:liecrochet} are fulfilled. Moreover, we have
\begin{equation}
    \label{eq:computationcrochetGraph}
    \frac{1}{2} {\rm ad}^2_S(H_0) = - (\partial_x \varphi)^2.
\end{equation}

Assume that $Q_1, \dots, Q_m \in C^0(\mathcal G;\R)$ and the restrictions of $Q_1, \dots, Q_m$ to each edge of $\cal G$ are  $C^{\infty}$ up to the boundary. Recalling the notations \eqref{eq:defH0subspacegraph}, \eqref{eq:defHNgraph}, and \eqref{eq:defHinftygraph}, we have the following result.  
\begin{theorem}
\label{thm:graph}
For every $\psi_0 \in L^2(\mathcal G)$, we have
\begin{equation*}
    \{e^{i\phi}\psi_0\mid\phi\in\mathcal{H}_\infty\}\subset \overline{\mathcal R_{0}(\psi_0)}.
\end{equation*}
\end{theorem}
The proof of the theorem is the same as the one of \Cref{thm:manifoldirichletneumann}. Finally, note that \Cref{thm:graph} directly implies \Cref{th:smalltimeapproximatequantumgraphs}.

\bibliographystyle{alpha}
\bibliography{SchrodingerBilinear}

\end{document}